\documentclass[12pt]{amsart}
\usepackage{amssymb, amsmath, mathtools, enumerate}
\newtheorem*{theorema}{Theorem A}
\newtheorem*{theoremb}{Theorem B}
\newtheorem*{theoremc}{Theorem C}

\usepackage{mabliautoref}
\usepackage[all]{xy}
\usepackage[dvipsnames]{xcolor}

\usepackage{float} %important

\usepackage[utf8]{inputenc}
\usepackage[T1]{fontenc}

\usepackage{marginnote}

\usepackage{amscd}

\usepackage{tikz}

\usepackage{mathrsfs}

\usepackage{amsopn}

\usepackage{upgreek} % need for uptau

\usepackage{mathtools, geometry, fullpage}

\newtheorem*{Theorem*}{Theorem}
\newtheorem*{Claim*}{Claim}
\newtheorem*{Question*}{Question}

\theoremstyle{remark}

\DeclareMathOperator{\Hom}{Hom}
\DeclareMathOperator{\Spec}{Spec\,}

\def\a{\mathfrak{a}}

\newcommand{\vp}{\varphi}
\newcommand{\sO}{\mathcal{O}}
\newcommand{\bR}{\mathbb{R}}
\newcommand{\bZ}{\mathbb{Z}}
\newcommand{\bN}{\mathbb{N}}
\newcommand{\bC}{\mathbb C}

\newcommand{\sC}{\mathcal{C}}
\newcommand{\sD}{\mathcal{D}}
\renewcommand{\:}{\colon}

\newcommand{\xalphaqn}{\prod_{m=1}^n \prod_{i=0}^d x_{i,m}^{\alpha_{i,m}/q}}
\newcommand{\xaqn}{\prod_{m=1}^n \prod_{i=0}^d x_{i,m}^{a_{i,m}/q}}

\usepackage{color}

\newcommand{\braket}[1]{\langle #1 \rangle}

\usepackage{amsfonts, mathrsfs}
\usepackage[cal=boondox, calscaled=1.05]{mathalfa}
\newcommand{\kay}{\mathcal{k}}

\newcommand{\cf}{{\itshape cf.} }

\renewcommand{\a}{\alpha}
\renewcommand{\epsilon}{\varepsilon}
\renewcommand{\bar}{\overline}

\title{Symbolic and Ordinary Powers of Ideals in Hibi Rings}
\author{Janet Page}
\email{jpage8@uic.edu}
\address{Department of Mathematics\\ University of Illinois at Chicago\\Chicago\\  IL 60607}
\author{Daniel Smolkin}
\email{smolkin@math.utah.edu}
\address{Department of Mathematics\\ University of Utah\\ Salt Lake City\\ UT 84112}
\author{Kevin Tucker}
\email{kftucker@uic.edu}
\address{Department of Mathematics\\ University of Illinois at Chicago\\Chicago\\  IL 60607}
\thanks{
The first author was partially supported by a Dean's Scholarship from the University of Illinois at Chicago.
The second author was partially supported by NSF grants DMS \#1252860, DMS \#1501102, and DMS \#1246989.
The third author is grateful to the NSF for partial support under Grants DMS \#1602070 and \#1707661, and for a fellowship from the Sloan Foundation. }

\usepackage{hyperref}
\usepackage{bm}

\begin{document}

\begin{abstract} 
We exhibit a class of Hibi rings which are diagonally F-regular over fields of positive characteristic, and diagonally $F$-regular type over fields of characteristic zero, in the sense of \cite{SmolkinCarvajal}.  It follows that such Hibi rings satisfy the uniform symbolic topology property \emph{effectively} in all characteristics. Namely, for rings $R$ in this class of Hibi rings, we have $\mathfrak p^{(dn)} \subseteq \mathfrak{p}^n$ for all $\mathfrak p \in \Spec R$, where $d = \dim(R)$. Further, we demonstrate that \emph{all} Hibi rings over fields of positive characteristic are 2-diagonally $F$-regular, and that the simplest Hibi ring not contained in the above class is not 3-diagonally $F$-regular in any characteristic. The former implies that $\mathfrak p^{(2d)} \subseteq \mathfrak p^2$ for all $\mathfrak p \in \Spec R$.
\end{abstract}

\maketitle

\section{Introduction}
In \cite{ELSsymbolicPowers}, Ein, Lazarsfeld, and Smith show that regular affine $\bC$-algebras satisfy the \emph{Uniform Symbolic Topology} property (USTP). Concretely, given a regular affine $\bC$-algebra $R$ of dimension $d$, they show that 
\begin{equation}
  \mathfrak p^{(dn)} \subseteq \mathfrak p^n \tag{$\star$} \label{eq:mainContainment}
  \end{equation}
  for all $\mathfrak p \in \Spec R$ and all $n > 0$, where $\mathfrak p^{(m)}$ denotes the $m$th symbolic power of $\mathfrak p$. %Their argument proved to be quite powerful and was adapted to positive characteristic in \cite{HaraACharacteristicPAnalogOfMultiplierIdealsAndApplications} and to mixed characteristic in \cite{MaSchwede}.
%new
By that point, Swanson had shown in \cite{SwansonSymbolicPowers} that for any prime ideal $\mathfrak p$ in a regular ring\footnote{More precisely, Swanson's result holds in any Noetherian ring for any ideal $I$ such that the $I$-symbolic and $I$-adic topologies are equivalent, though it is now known to hold for any prime ideal in any complete local domain; see \cite[Proposition 2.4]{HunekeKatzValidshtiUSTPFiniteExtensions}}, there exists a number $h$, depending on $\mathfrak p$, so that $\mathfrak p^{(hn)} \subseteq \mathfrak p^n$ for all $n$. Thus the most striking aspect of Ein--Lazarsfeld--Smith's theorem is its uniformity---essentially, that we could take $h$ independent of $\mathfrak p$. Later, Huneke, Katz, and Validashti showed that a similar uniformity holds in a large class of isolated singularities \cite{HunekeKatzValidashtiUniformEquivIsolatedSing}. This lead them to ask,

\begin{Question*}[\cite{HunekeKatzValidashtiUniformEquivIsolatedSing}]
Let $(R, \mathfrak m)$ be a complete local domain. Does there exist $h \geq 1$ such that $\mathfrak p^{(hn)} \subseteq \mathfrak p^n$ for all prime ideals $\mathfrak p$ and all $n\geq 1$?
\end{Question*}
\noindent In case the answer is affirmative, it is also desirable to have an explicit description of the number $h$. So far the answer is known in very few cases, including 2-dimensional rational singularities \cite{WalkerRationalSingSymbolicPowers}, invariant subrings of finite group actions on regular rings \cite{HunekeKatzValidshtiUSTPFiniteExtensions}, and Segre products of polynomial rings \cite{SmolkinCarvajal}; see \cite{DaoDestefaniGrifoHunekeSymbolicSurvey} for an excellent survey on this subject.  In this paper, we add to the story by answering the question in the affirmative for a large class of rings. 
% /new

%Until recently, however, the argument was not generalized outside of the regular setting. Indeed, there are few examples of non-regular rings for which we know an explicit number $h$ such that $\mathfrak p^{(hn)} \subseteq \mathfrak p^n$ for all $n$ and all prime ideals $\mathfrak{p}$ \cite{HunekeKatzValidshtiUSTPFiniteExtensions, WalkerRationalSingSymbolicPowers}; see also \cite{WalkerUniformHarbourneHuneke, WalkerUniformSymbolicMultinomial, WalkerUniformSymbolicNormalToric, HunekeKatzValidashtiUniformEquivIsolatedSing} and see \cite{DaoDestefaniGrifoHunekeSymbolicSurvey} for an excellent survey on this subject.

We follow the approach of \cite{SmolkinCarvajal}. There, Carvajal-Rojas and the second named author describe class of rings in positive characteristic, called \emph{diagonally $F$-regular} rings, in which one can run a variant of Ein--Lazarsfeld--Smith's argument. Thus, prime ideals in a diagonally $F$-regular ring of dimension $d$ satisfy the containment \eqref{eq:mainContainment} for all $n \geq 1$. Using standard reduction mod $p$ techniques (\textit{cf.} \cite{HochsterHunekeTightClosureInEqualCharactersticZero}), their results can be used to show \eqref{eq:mainContainment} in characteristic 0. For instance, a toric ring over any field satisfies \eqref{eq:mainContainment} if it is diagonally $F$-regular over every field of sufficiently large characteristic. 

In this paper, we investigate which Hibi rings are diagonally $F$-regular. A Hibi ring is a kind of toric ring which can be associated to any finite poset.  These rings were first introduced by Hibi in \cite{Hibi}, who showed that they are Gorenstein if and only if the associated poset is pure.  We find an infinite class of Hibi rings which are diagonally F-regular. Our main result is the following.
\begin{theorema}[\textit{cf.} \autoref{thm:comparabletopnodes}]
Let $P$ be a finite poset and $\kay$ a field of positive characteristic. If the set of top nodes of $P$ is completely ordered, then the Hibi ring over $\kay$ associated to $P$, denoted  $\kay[\mathcal I(P)]$, is diagonally $F$-regular. It follows that $\mathfrak p^{(hn)} \subseteq \mathfrak p^n$ for all prime ideals $\mathfrak p$ of height $h$ in $L[\mathcal I (P)]$, where $L$ is any field. In particular, \eqref{eq:mainContainment} holds for all prime ideals $\mathfrak p$ in such rings.
\end{theorema}

\noindent 
In particular, as Segre products of polynomial rings are also the Hibi rings associated to a poset with two incomparable chains of elements (and hence have no top nodes), we recover \cite[Theorem 5.6]{SmolkinCarvajal}.  The above result also gives many new examples of singular toric rings which are diagonally $F$-regular in arbitrarily large dimension---the dimension of a Hibi ring is one more than the number of elements in the associated poset. More generally, our result supports the philosophy that algebraic and geometric properties of Hibi rings can often be read off directly from properties of the associated poset.

Diagonal $F$-regularity of a ring of positive characteristic is defined in terms of the properties of the set of $p^{-e}$-linear (or $F$-splitting type) maps which lift over the $n$-fold diagonal embedding for all positive integers $n$.  In particular, in the language of \cite{SmolkinCarvajal}, a ring is diagonally $F$-regular if and only if it is $n$-diagonally $F$-regular for all $n$.
In \cite[Section 6]{Smolkin}, the second author investigated the combinatorial conditions required of affine toric rings to be 2-diagonally $F$-regular.  In particular, from this characterization it was previously known that there are affine toric rings that are not 2-diagonally $F$-regular.  We are able to show this is not the case for Hibi rings.

\begin{theoremb}[\textit{cf.} \autoref{thm:hibi2dfr}]
All Hibi rings are 2-diagonally $F$-regular in positive characteristic. It follows that $\mathfrak p^{(2h)} \subseteq \mathfrak p^2$ for any prime ideal $\mathfrak p$  of height $h$ in any Hibi ring. 
\end{theoremb}

In light of the above theorem, it is natural to ask whether the restrictions on the poset in the hypothesis of Theorem A are actually required, or rather if all Hibi rings might be diagonally $F$-regular.  Indeed, there are other reasons one might expect this to hold.  It is shown in \cite[Proposition 5.2]{SmolkinCarvajal} that the class group of a (local) diagonally $F$-regular ring must be torsion free, a property that all Hibi rings have but that the more general class of affine toric rings do not.  Nonetheless, we are also able to demonstrate that the Hibi ring of the simplest poset with incomparable top nodes is not diagonally $F$-regular, suggesting that the characterization of diagonally $F$-regular Hibi rings in Theorem A may be exhaustive. 
\begin{theoremc}
[\textit{cf.} \autoref{ex:not3diagFreg}]
There exists a four element poset $P$ for which the associated Hibi ring is not 3-diagonally $F$-regular in any characteristic.
\end{theoremc}

\noindent
In order to arrive at our example, it is necessary to build on \cite[Theorem 6.4]{Smolkin} and make explicit the combinatorial conditions required of an affine toric ring to be $n$-diagonally $F$-regular for $n > 2$.  The techniques used here are common in the study of toric rings and the proof follows along the same lines as \cite[Section 6]{Smolkin}.  As such, and also because there is nothing involved is specific to the Hibi rings central to our main results, we have relegated this material to an appendix.  Unlike the case where $n=2$, there does not seem to be an immediate interpretation of the associated combinatorial condition in terms of tiling conditions, and so it seems to be far more difficult to verify in practice.

To our knowledge, the above example is the first to show that a 2-diagonally $F$-regular ring need not be diagonally F-regular, even for affine toric rings.  It remains an open question what additional properties one might impose to guarantee that an affine toric ring is diagonally $F$-regular.  From our example, we conclude that it is not sufficient to ask for 2-diagonal $F$-regularity and torsion free class group, or even for the $F$-signature to be somewhat large (11/30) with terminal singularities.

The authors would like to thank Javier Carvajal-Rojas and Karl Schwede for helpful conversations and support related to this project.

\section{Background}
Here we explain some preliminaries on Cartier algebras and Hibi rings. In this section, we use the following setting: 
\begin{setting}
$\kay$ is a field of characteristic $p > 0$. Given $e \geq 0$, we use the shorthand $q \coloneqq p^e$. $R$ is a domain and a $\kay$ algebra essentially of finite type. 
\label{setting:bg}
\end{setting}
  As $R$ is a domain, we can define $R^{1/q}$ to be the set of $q^{\textrm{th}}$ roots of elements of $R$ in a fixed algebraic closure of $\operatorname{frac}(R)$. As $R$ has characteristic $p$, each element $r\in R$ has a \emph{unique} $q^\textrm{th}$ root $x^{1/q}$ in $R^{1/q}$. This means that $R^{1/q}$ is canonically isomorphic to $R$ \emph{as a set}, though not as an $R$-module. In any case, this means that $(\cdot)^{1/q}$ is a functor: given a map $\vp: R \to S$, we define 
  \[
   \vp^{1/q}: R^{1/q} \to S^{1/q}, \quad \vp^{1/q}(r^{1/q}) \coloneqq \vp(r)^{1/q}
  \]
  
  An object of intense study in the field of positive characteristic commutative algebra is the set of maps $\Hom_R(R^{1/q}, R)$. We equip this set with left and right $R$-module structures. Namely, for $r\in R$ and $\vp \in \Hom_R(R^{1/q}, R)$, we define $(r \cdot \vp)(x) \coloneqq r\vp(x)$, and $(\vp \cdot r)(x) \coloneqq \vp(r^{1/q}x)$. Note that $r \cdot \vp = \vp \cdot r^q$. 
  
  \begin{definition}
  By a \emph{$p^{-e}$-linear map on $R$}, we mean an element of $\Hom_R(R^{1/q}, R)$.
  \end{definition}
  The above terminology comes from the fact that an element of $\Hom_R(R^{1/q}, R)$ is the same as a map of Abelian groups $\vp: R \to R$ satisfying the \emph{$p^{-e}$-linearity} property, namely $\vp(r^{p^e}x) = r \vp (x)$.

\begin{definition}[Cartier Algebras]
Work in \autoref{setting:bg}. A \emph{Cartier algebra} $\sC$ over $R$ (or \emph{Cartier $R$-algebra}) is an $\mathbb{N}$-graded unitary  ring\footnote{Not necessarily commutative.} $\bigoplus_{e\in \bN} \sC_e$   such that $\sC_0 = R$, equipped with a graded finitely generated $R$-bimodule structure so that $a \cdot \kappa = \kappa \cdot a^q$ for all $\kappa \in \sC_e$. Note that, strictly speaking, $\sC$ is not an $R$-algebra, as $R$ is not in the center of $\sC$. 
\end{definition}
Any ring in our setting comes with a canonical choice of Cartier algebra, $\sC_R$. Namely, we define
\[
   \sC_{e, R} \coloneqq \Hom_R(R^{1/q}, R)
\]
for any $e\geq 0$. Then the Cartier algebra $\sC_R \coloneqq \bigoplus_{e \in \bN} \sC_{e, R}$ is called the \emph{full Cartier algebra on $R$}. We define multiplication in $\sC$ by setting
\[
  \varphi \cdot \psi \coloneqq \varphi \circ \left(  \psi^{1/p^d} \right)
\]
for any $\varphi \in \sC_{d, R}$ and $\psi \in \sC_{e, R}$. Multiplication is defined for non-homogeneous elements of $\sC$ by distributivity. Throughout this work, we will only be concerned with subalgebras on $R$ contained in $\sC$. Thus,  we are only concerned with Cartier algebras in the sense of \cite{SchwedeNonQGorTestIdeals}.

The Cartier algebras we are mainly interested in are the so-called \emph{diagonal Cartier algebras}. These are the Cartier algebras first defined in \cite{Smolkin} and \cite{SmolkinCarvajal} in order to get a subadditivity formula on non-regular rings. 

\begin{definition}[{\cite{SmolkinCarvajal}, \cf \cite[Notation 3.7]{Smolkin}}] \label{def.diagCartAlg}
Let $R$ be a $\kay$-algebra. For $n\in \bN$ we define the \emph{$n$-th diagonal Cartier algebra of $R/\kay$}, denoted by $\mathcal{D}^{(n)}(R)$, as follows. In degree $e$,  $\sD^{(n)}_{e}(R) \subset \sC_{e,R}$ consists of $R$-linear maps $\varphi \: R^{1/q} \to R$ such that there is an $R$-linear lifting $\widehat{\varphi}\:  (R^{\otimes n})^{1/q} \to R^{\otimes n}$ making the following diagram commutative:
\begin{equation*} 
\xymatrix{
(R^{\otimes n})^{1/q} \ar[r]^-{\widehat{\varphi}} \ar[d]_-{\Delta_n^{1/q}} & R^{\otimes n}\ar[d]^-{\Delta_n}\\ 
 R^{1/q} \ar[r]^-{\varphi} & R
}
\end{equation*}
Here, $\Delta_n$ is the canonical multiplication map. We say that  a map $\widehat \vp$ as above is \emph{compatible with the diagonal}. 
\end{definition}

It is straightforward to verify $\sD^{(n)}(R)$ is a Cartier subalgebra of $\sC_R$, see \cite[Proposition 3.2]{Smolkin} and \cite[Definition 2.10]{Fsigpairs1}. When the ring $R$ is clear from context, we will refer to this Cartier algebra simply as $\sD^{(n)}$. 
\begin{definition}
Let $n \geq 1$. A ring $R$ as in \autoref{setting:bg} is said to be \emph{$n$-diagonally $F$-regular} if the following holds: for all $x\in R$ there exists some $e$ and some $\vp \in \sD^{(n)}_e(R)$ with $\vp(x^{1/q}) = 1$
\end{definition}

It turns out we don't need to check all $x$ in the above definition. The following proposition tells us we only need to check a single $x$ if we have some more information about $x$:

\begin{proposition}
Let $R$ be as in \autoref{setting:bg} and suppose $y \in R$ is an element with $R_y$ smooth.\footnote{For instance, if $\kay$ is perfect then this is the same as saying $R_y$ is regular.} Then $R$ is $n$-diagonally $F$-regular if there exists some $e$ and some $\vp \in \sD^{(n)}_e(R)$ with $\vp(y^{1/q}) = 1$
\end{proposition}
\begin{proof}
This follows by combining \cite[Proposition 2.4]{SmolkinCarvajal} and \cite[Proposition 5.1]{SmolkinCarvajal}.
\end{proof}

\begin{definition}
A ring $R$ as in \autoref{setting:bg} is said to be \emph{diagonally $F$-regular} if $R$ is $n$-diagonally $F$-regular for all integers $n \geq 1$.
\end{definition}

Diagonally $F$-regular rings are precisely the rings in which we can run Ein--Lazarsfeld--Smith's argument using the subadditivity theory of \cite{Smolkin, SmolkinCarvajal}. Thus: 

\begin{proposition}[{\cite[Theorem 4.1]{SmolkinCarvajal}}]
Let $R$ be as in \autoref{setting:bg} and suppose that $R$ is diagonally $F$-regular. Let $\mathfrak p \subseteq R$ be a prime ideal of height $h$. Then we have $\mathfrak p^{(hn)} \subseteq \mathfrak p^n$ for all $n \geq 1$. 
\end{proposition}

% diagonally f regular if we can split an element in the jacobian ideal

We will construct examples of Hibi rings, which are toric rings defined by finite posets.  We recall the following definitions:

\begin{definition}
Given a poset $P$, a poset ideal $I \subset P$ is a subset of $P$ such that if $x \in I$ and $y \leq x$, then $y \in I$.  We denote the set of all poset ideals $\mathcal{I}(P)$.
\end{definition}

Given a poset with $n$ elements, the associated Hibi ring over $\kay$ is a subring of $\kay[t,x_1, \dots, x_n]$ defined by adjoining to $\kay$ a monomial for each poset ideal in $P$.

\begin{definition}
Given a finite poset $P = \{v_1, \dots, v_n\}$, the associated Hibi ring is:
\[
\kay[\mathcal{I}(P)]:= \kay\left[t\prod_{v_i \in I}x_i \Bigg| I \text{ is a poset ideal in P}\right]
\]
\end{definition}

We will denote by $\bar{P}$ the poset which is formed from $P$ by adding a unique minimal element $-\infty:= v_0$ and will identify this with $x_0:= t$.  In general, when drawing the poset associated to a Hibi ring, we will draw $\bar{P}$ rather than $P$. Note that  $x_0^{a_0} \cdots x_n^{a_n} \in \kay[\mathcal I(P)]$ if and only if $a_j \leq a_i$ whenever $v_i \lessdot v_j \in \bar{P}$.  

\section{Diagonal Cartier Algebras on Toric Varieties}
In \cite[Section 6]{Smolkin}, the second-named author finds a combinatorial description of the Cartier algebra $\sD^{(2)}(R)$, where $R$ is a toric ring in positive characteristic. As a consequence, the sets $\sD^{(2)}_e(R)$ are generated, as $\kay$-vector spaces, by maps sending monomials to monomials.  Before proceeding, we wish to show the analogous result for the Cartier algebras $\sD^{(n)}(R)$. 
\begin{setting}
  We let $\Sigma$ be a rational cone in $\bR^d$ and $R = R(\Sigma)$ the associated toric ring over a perfect field $\kay$ of characteristic $p > 0$. Set $X = \Spec R$. For all rays $\rho \subseteq \Sigma$, we let  $v_\rho$ denote the primitive generator of $\rho$. That is, $v_{\rho}$ is the shortest nonzero vector in $\bZ^d \cap \rho$.
 \label{setting:toric}
\end{setting}
Choose $e > 0$ and set $q = p^e$. Let $\mathcal T  = \kay[x_1^\pm, \ldots, x_d^\pm]$ be the coordinate ring of the $d$-dimensional torus over $\kay$. Then for each $a \in \frac{1}{q}\bZ^n$, we have a map $\pi_a: \mathcal T^{1/q} \to \mathcal T$ given by
  \[
    \pi_a(x^u) = \begin{cases}
      x^{a+u}, & a+u\in \bZ^d\\
      0, & \textrm{otherwise}
    \end{cases}
  \]
  \begin{lemma}[{\cite[Lemma 4.1]{PayneFrobeniusSplitToric}}]
    The set $\{ \pi_a \mid a \in \frac{1}{q}\bZ^d \}$ forms a $\kay$-vector space basis of  $\Hom_{\mathcal T}(\mathcal T^{1/q}, \mathcal T)$.
  \end{lemma}
%   \begin{definition}
%     Let $R$ be a toric ring. By a \emph{monomial map} $R^{1/p^e} \to R$ we mean a map $\pi_a$ for some $a \in 1/p^e \bZ$.
%   \end{definition}
  For each \emph{torus-invariant} divisor $D$ on $X$, one can associate a polytope $P_D \subseteq \bR^d$. In the case of $D= -K_X$, the associated polytope is called the \emph{anticanonical polytope} of $X$, and is given by
  \[
    P_{-K_X} = \left\{ u \in \bR^d \mid \langle u, v_\rho \rangle  \geq -1 \textrm{ for all rays } \rho \subset \Sigma \right\}.
  \]
  We set  $P_{-K_X}^\circ$ to be the interior of $P_{-K_X}$. We have the following:
  \begin{proposition}[{\cite[Proposition 4.3]{PayneFrobeniusSplitToric}}]
    The set $\{ \pi_a \mid a \in \frac{1}{q}\bZ^d\cap P_{-K_X}^\circ \}$ forms a basis of  $\Hom_{R}(R^{1/q}, R)$ as a $\kay$-vector space. 
    \label{Prop:PayneProp}
  \end{proposition}
  
The following is a characterization of the Cartier algebras $\sD^{(n)}(R)$ for a toric ring $R$. The proof is essentially the same as \cite[Theorem 6.4]{Smolkin}; we leave it for the appendix. 
\begin{theorem}[{C.f.~\cite[Theorem 6.4]{Smolkin}}]
  Work in \autoref{setting:toric} Set $n, e> 0$. Then $\sD^{(n)}_{e}(R)$ is generated as a $\kay$-vector space by the maps  $\pi_a \in \Hom_{R}(R^{1/q}, R)$ satisfying the following: for all sets of vectors $u_1, \ldots, u_n \in \frac{1}{q}\bZ^d$, there exist $v_i \in u_i + \bZ^d$ such that $v_i \in P_{-K_X}^\circ$ and $\sum_{i=1}^n v_i \in a - P_{-K_X}^\circ$. 
\label{thm:toricDn}
\end{theorem}

The following consequence of \autoref{thm:toricDn} will be useful to us later. 

\begin{corollary}
Work in \autoref{setting:toric}. Let $\mathbf x^{a}$ be a monomial in $R$. If there is a map $\vp\in \sD^{(n)}_e(R)$ with $\vp(\mathbf x^{a/p^e}) = 1$, then $\pi_{-a/p^e} \in \sD^{(n)}_e$. 
\label{thm:monomialSplittings}
\end{corollary}
\begin{proof}
By \autoref{thm:toricDn}, we may write $\vp = \sum_i c_i \pi_{a_i}$ where $a_i \in \frac{1}{p^e}\bZ$ and $\pi_{a_i} \in \sD^{(n)}$ for all $i$. Then 
\[
  \vp\left( \mathbf x^{a/p^e} \right) = \sum_{a_i \in a/p^e + \bZ^d} c_i \mathbf x^{a_i + a/p^e} = 1.
\]
It follows that $a_i + a/p^e = 0$ for some $i$, so $\pi_{-a/p^e} \in \sD^{(n)}_e$.
\end{proof}

The following consequences of \autoref{thm:toricDn} are of independent interest. 

\begin{corollary}
Work in \autoref{setting:toric}. Then we have
\[
 \sD^{(n)}(R) \supseteq \sD^{(n+1)}(R)
\]
for all $n > 0$. In particular, if a toric ring is $n$-diagonally $F$-regular, then it is $m$-diagonally $F$-regular for all $m \leq n$. 
\end{corollary}
\begin{proof}
  Fix some $e>0$, and set $q = p^e$. Let $\pi_a\in \sD^{(n+1)}_{e}(R)$ and let $u_1, \ldots, u_n \in \frac{1}{q}\bZ^d$ be arbitrary.  Setting $u_{n+1} = 0$,  \autoref{thm:toricDn} tells us that there exist $v_i \in u_i + \bZ^d$, for $1 \leq i \leq n+1$, such that $v_i \in P_{-K_X}^\circ$ and $\sum_{i=1}^{n+1} v_i \in a - P^\circ_{-K_X}$. By definition, this means $\braket{v_{n+1}, v_\rho} > -1$ for all $\rho$. As $v_{n+1}, v_\rho \in \bZ^d$, this means  $\braket{v_{n+1}, v_\rho} \geq 0$.  It follows that $P^\circ_{-K_X}  + v_{n+1} \subseteq P^{\circ}_{-K_X}$. Thus we have
  \begin{equation*}
    \sum_{i=1}^n v_i \in a - P_{-K_X}^\circ - v_{n+1}  = a - (P_{-K_X}^\circ + v_{n+1}) \subseteq a - P_{-K_X}^\circ.
  \end{equation*}
  By \autoref{thm:toricDn}, it follows that $\pi_a \in \sD^{(n)}_{e}(R)$, as desired. 
\end{proof}

\begin{theorem}
Let $P$ be a finite poset and $R \coloneqq \kay [\mathcal I(P)]$ the associated Hibi ring over the field $\kay$. Then $R$ is $2$-diagonally $F$-regular over $\kay$.
\label{thm:hibi2dfr}
\end{theorem}
\begin{proof}
Let $n = | P|$ and write $P = \{v_1, \ldots, v_n\}$. Set $v_0 = -\infty$. Then $R$ is a normal affine $(n+1)$-dimensional toric variety. The rays of the cone corresponding to $R$ are given by the vectors, $e_i - e_j$ for all $0 \leq i, j \leq n$ such that $v_i \lessdot v_j$, along with $e_i$ for all $i$ such that $v_i$ is maximal in $P$. It follows that $P^\circ_{-K_R}$ is given by the linear inequalities:
\[
P^\circ_{-K_R} = 
\left\{ (x_0, \ldots, x_n) \Big | \begin{array}{cc}
     x_i - x_j >  -1, & \textrm{whenever } v_i \lessdot v_j \\
     x_i > -1, & \textrm{whenever } v_i \textrm{ maximal}
\end{array} \right \}
\]

It suffices to show that, for all $a \in \bZ^d$, there exists $e \gg 0$ such that $\pi_{a/p^e} \in \sD^{(2)}$. By \cite[Theorem 6.4]{Smolkin}, it suffices to find $\varepsilon > 0$ such that 
\[
  P^\circ_{-K_R} \cap \left( \delta  - P^\circ_{-K_R} \right)
\]
covers all of $\bR^{n+1}$ under integer translations whenever $\delta \in \bR^{n+1}$ is a vector with $| \delta | < \varepsilon$. In turn, it suffices to find some $\varepsilon > 0$ such that the set
\[
S \coloneqq \left\{ (x_0, \ldots, x_n) \Big | \begin{array}{cc}
     -1 < x_i - x_j < 1 - \varepsilon, & \textrm{whenever } v_i \lessdot v_j \\
     -1 < x_i < 1 - \varepsilon, & \textrm{whenever } v_i \textrm{ maximal}
\end{array} \right \}
\]
covers $\bR^{n+1}$ under integer translations. We will show that $\varepsilon = 1/3(n+1)$ works.  

To see this, let $(x_0, \ldots, x_n) \in \bR^{n+1}$ be arbitrary and let $\vp: \bR \to \bR/\bZ$ be the canonical surjection. Then, thinking of $\bR/\bZ$ as a circle of unit circumference, $\vp$ is a local isometry. The points $\vp(x_i)$ partition $\bR/\bZ$ into $n+1$ arcs. By the pigeon-hole principle, one of these arcs must have length greater than or equal to $1/(n+1)$. Fix a point $\xi$ in the mid-point of such an arc and choose a point 
\[
 \zeta \in \left( -1-\frac{1}{2(n+1)},-\frac{1}{2(n+1)} \right]
 \]
 with $\vp(\zeta) = \xi$. Let $\psi: \bR/\bZ \to [\zeta, \zeta + 1)$ be the map such that $\vp \circ \psi$ is the identity. It follows that $\psi \circ \vp(x_i) \equiv x_i$ mod $\bZ$ for each $i$, that 
 \[
 \left| \psi\circ \vp(x_i) - \psi \circ \vp(x_j) \right| \leq 1 - \frac{1}{2(n + 1)}
 \]
 for all $i,j$, and that $-1 < \psi \circ \vp(x_i) \leq 1 - 1/2(n+1)$ for all $i$. Thus, $(\psi\circ \vp(x_0), \ldots, \psi \circ \vp(x_n))$ is an integer translation of $(x_0, \ldots, x_n)$ in $S$.
\end{proof}

The proof of \cite[Theorem 4.1]{SmolkinCarvajal} actually shows the following: if $R$ and $m$-diagonally $F$-regular, then $\mathfrak p^{(mh)} \subseteq \mathfrak p^m$ for all prime ideals $\mathfrak p \subseteq R$, where $h = \operatorname{height}(\mathfrak p)$. Thus we get:

\begin{corollary}
Let $R$ be a Hibi ring over a field. Then $\mathfrak p^{(2h)} \subseteq \mathfrak p^2$ for all  $\mathfrak p \in \Spec R$ with $h = \operatorname{height}(\mathfrak p)$.
\end{corollary}
\noindent
The characteristic 0 case of the above corollary follows from the positive characteristic case by a standard argument; see \cite[Section 6]{SmolkinCarvajal}, \cite{HochsterHunekeTightClosureInEqualCharactersticZero}.

\section{The main combinatorial problem}
%\Janet{Working on this -- Sept 11; done, but it probably could still be cleaned up...}
%In this section, we describe a combinatorial problem equivalent to checking the $n$-diagonal F-regularity of a Hibi ring.

Let $R = \kay[\mathcal{I}(P)]$ be a Hibi ring associated to a poset $P = \{v_1,\dots,v_d\}$, and let $S = \kay[x_0, \dots , x_d]$ so that there is a natural inclusion $R \subset S$.  We can reduce checking that $R$ is diagonally $F$ regular to a combinatorial property.
We begin by fixing the element $z = \prod_{I \in \mathcal{I}(P)} x_I = \prod_{i=0}^d x_i^{r_i}$ of $R$.  The choice of this element is not particularly important, only that $r_i > r_j$ whenever $v_i \lessdot v_j$.  From this, it follows that $z$ vanishes along every torus invariant divisor of $\Spec(R)$, so that $R[1/z]$ is an algebraic torus over $k$ and hence a regular ring.  Following the strategy of \cite[Section 5.1]{SmolkinCarvajal}, it suffices to find $e > 0$ and $\phi \in \sD^{(n)}_e(R)$ for all $n > 0$ with $\phi(z^{1/q}) = 1$.  To that end, let us fix hereafter some $e > 0$ sufficiently large that $q = p^e > \max_{i}r_i$ is strictly larger than any exponent appearing in $z$.  We will show that such a map is equivalent to a choice of $\delta_{i,m}$ satisfying a particular property for every choice of integers $\alpha_{i,m}$ in $[0, q-1]$ such that $\sum_{m=1}^n \alpha_{j, m} \equiv r_j$ (mod $q$) for all $j$.

%\Janet{Dan/Kevin: Is this how we want to phrase this?}
\begin{proposition}\label{Prop:combinatorialproperty}
  Let $R, S$ and $z$ be as above. Then there exists $\phi \in \sD_e^{(n)}(R)$ with $\phi(z^{1/q}) = 1$ if and only if the following statement holds: for $0 \leq i \leq d$ and $1 \leq m \leq n$, let $\alpha_{i, m}$ be integers in $[0, q-1]$ such that $\sum_{m=1}^n \alpha_{j, m} \equiv r_j$ (mod $q$) for all $j$. Set $N_j = \lfloor \sum_{m=1}^n \frac{\alpha_{j,m}}{q} \rfloor$. For all $i, j,$ and $m$, let $\varepsilon_{j, i, m} = 1$ if $\alpha_{j,m} > \alpha_{i, m}$ and let $\varepsilon_{j, i, m} = 0$ otherwise. Then there exist $\delta_{i, m} \in \bZ$ with 
  \begin{enumerate}
    \item $\delta_{i,m} \geq 0$ for all $m$ whenever $v_i$ is maximal in $P$,
    \item $\delta_{j,m} \leq \varepsilon_{j, i, m} + \delta_{i,m}$  for all  $m$ whenever  $v_i \lessdot v_j$,  and 
    \item $\sum_{m=1}^n \delta_{j,m} = N_j $
  \end{enumerate}
\end{proposition}

%\begin{Proposition}\label{Prop:combinatorialproperty}
%In the setting above, $R$ is $n$-diagonally $F$ regular if for every choice of $\alpha_{i,m}$ with $1 \leq i \leq d$ and $1 \leq m \leq n$, there exists a choice of $\delta_{i,m}$ satisfying:
% \begin{equation*} %\label{eq:deltaconditions}
% \begin{array}{lc}
%         0 \leq \delta_{j,m} \leq \lceil \frac{\a_{j,m}}{q} - \frac{\alpha_{i,m}}{q}\rceil + \delta_{i,m} & \mbox{ for all } m, \mbox{ and} \medskip \\ 
%          \sum_{m=1}^n \delta_{j,m} = \lfloor \sum_{m=1}^n \frac{\alpha_{j,m}}{q} \rfloor & 
%    \end{array}.
%\end{equation*}

%\end{Proposition}

%\subsection*{Proof strategy} 
\begin{proof}
%Note that $\eqref{eq:deltaconditionsa}, \eqref{eq:deltaconditionsb}$, and $\eqref{eq:deltaconditionsc}$ depend only on $\alpha_{i,m}$, and so the particular choices of $a_{i,m} = \lfloor a_{i,m}/q \rfloor + \alpha_{i,m}$ do not matter.

First suppose that for every choice of %$\xaq \in S^{1/q}$ 
$\alpha_{i,m}$ as above, we have such a choice of $\delta_{i,m}$.  
In order to construct the requisite map $\phi$, we will exhibit a $p^{-e}$-linear map $\widetilde{\phi}$ on $S$ and check that $\widetilde\phi(R^{1/q})\subseteq R$.  The advantage here is that defining $p^{-e}$-linear maps on $S$ is straightforward as $S^{1/q}$ is a free $S$-module, and we can then set $\phi$ to be the restriction of $\widetilde\phi$ to $R$.  Similarly, in order to show that $\phi \in \sD^{(n)}_e(R)$, we need to check that $\phi$ admits a lifting to a $p^{-e}$-linear map $\phi_n$ on $R^{\otimes n}$ compatible with the diagonal for each $n > 0$.  To do this, we will again construct a $p^{-e}$-linear map $\widetilde\phi_n$ on $S^{\otimes n}$ and verify $\widetilde\phi_n((R^{\otimes n})^{1/q})\subseteq R^{\otimes n}$. The map $\phi_n$ will be the restriction of $\widetilde\phi_n$ to $R^{\otimes n}$, and verifying the requisite compatibility with the diagonal completes the proof.

{
\renewcommand{\a}{\alpha}
\subsection*{Construction of \mbox{$\phi$}}
    We have that $S^{1/q}$ is a free $S$-module with free basis given by the monomials $x_0^{\a_0/q} \cdots x_d^{\a_d/q}$ with $\a_i \in \mathbb{Z}$ and $0 \leq \a_i < q$ for each $i = 0, \ldots, d$. In particular, $z^{1/q}$ is an element of the monomial free basis, and we set $\widetilde\phi$ to be the projection onto that factor.  Thus, for an arbitrary monomial $x_0^{a_0/q} \cdots x_d^{a_d/q}$ with $a_i \in \mathbb{Z}_{\geq 0}$ of $S^{1/q}$, we have
    \[
    \widetilde{\phi}(x_0^{a_0/q} \cdots x_d^{a_d/q}) = \left\{
    \begin{array}{c@{\qquad}l}
    x_0^{(a_0-r_0)/q} \cdots x_d^{(a_d - r_d)/q} &  a_i \equiv r_i \mod q \mbox{ for all } i=0,\ldots,d\\
    0 & \mbox{otherwise}
    \end{array}\right..
    \]
    We need to check that $\widetilde{\phi}(R^{1/q})\subseteq R$.  We will do this monomial by monomial.  An arbitrary monomial of $R^{1/q}$ has the form $x_0^{a_0/q} \cdots x_d^{a_d/q}$ with $a_i \in \mathbb{Z}_{\geq 0}$ such that $a_i \geq a_j$ for all $v_i \lessdot v_j$.  If $a_i \not \equiv r_i \mod q$ for some $i= 0, \ldots, d$, then clearly $\widetilde\phi(x_0^{a_0/q} \cdots x_d^{a_d/q}) = 0 \in R$ and we are done. Else, suppose that $a_i \equiv r_i \mod q$ for all $i=0,\ldots,d$.  In this case, note in particular that $(a_i - r_i)/q = \lfloor a_i/q \rfloor$ for all $i$.  If $v_i \lessdot v_j$, then $a_i \geq a_j$ implies $a_i/q \geq a_j/q$ and hence also $\lfloor a_i/q \rfloor \geq \lfloor a_j/q \rfloor$.  Thus, it follows that
    \[
    \widetilde{\phi}(x_0^{a_0/q} \cdots x_d^{a_d/q}) = x_0^{\lfloor a_0/q \rfloor} \cdots x_d^{\lfloor a_d/q \rfloor} \in R
    \]
    as desired.  Setting $\phi = \widetilde{\phi}|_{R^{1/q}}$, we observe that $\phi(z^{1/q}) = 1$ by construction.

\subsection*{Construction of $\widetilde\phi_n$}
% It remains to show that $\phi \in \sD_e^{(n)}$ for all $n \in \mathbb{N}$.  To that end, fixing $n \in \mathbb{N}$ we need to construct a lift $\phi_n$ of $\phi$ to $R^{\otimes n}$ compatible with the diagonal. In other words, we require that the following diagram 
% \[
% \xymatrix{
% (R^{\otimes n})^{1/q} \ar[r]^-{\phi_n} \ar[d]_{\Delta_n^{1/q}} & R^{\otimes n} \ar[d]^{\Delta_n} \\
% R^{1/q} \ar[r]_{\phi} & R
% }
% \]
% is commutative.

Working towards the construction of $\phi_n$, we first define a $p^{-e}$-linear map $\widetilde{\phi}_n$ on $S^{\otimes n}$.
% in such a way that the diagram
% \[
% \xymatrix{
% (S^{\otimes n})^{1/q} \ar[r]^-{\widetilde\phi_n} \ar[d]_{\Delta_n^{1/q}} & S^{\otimes n} \ar[d]^{\Delta_n} \\
% S^{1/q} \ar[r]_{\widetilde\phi} & S
% }
% \]
% is commutative. 
Now, $(S^{\otimes n})^{1/q}$ is a free $S$-module with free basis given by the monomials $\prod_{m=1}^n \prod_{i=0}^d x_{i,m}^{\a_{i,m}/q}$ with $\a_{i,m} \in \mathbb{Z}$ and $0 \leq \a_{i,m} < q$ for each $i = 0, \ldots, d$ and $m = 1, \ldots, n$. We specify the image of such monomials under $\widetilde{\phi}_n$ in two cases. 
\begin{itemize}
    \item Case 1: $\sum_{m=1}^n \a_{i,m} \not \equiv r_i \mod q$ for some $1 \leq i \leq n$. 
    \item Case 2: $\sum_{m=1}^n \a_{i,m} \equiv r_i \mod q$ for all $1 \leq i \leq n$.  
\end{itemize}

In Case 1, simply set $\widetilde{\phi}_n(\prod_{m=1}^n \prod_{i=0}^d x_{i,m}^{\a_{i,m}/q}) = 0$.  Hereafter, we will assume we are working to specify the image under $\widetilde{\phi_n}$ of a monomial basis element in Case 2, and as before, set 
    \[
    N_i = \frac{(\sum_{m=1}^n \a_{i,m}) - r_i }{q} = \left\lfloor \frac{\sum_{m=1}^n \a_{i,m}}{q} \right\rfloor \geq 0
    \]
for each $i = 0, \ldots, d$.  Note that, while not immediate from the notation, $N_i$ depends heavily on the particular basis element in question. 

We have assumed we have a choice of $\delta_{i,m}$ for $0 \leq i \leq d$ and $1 \leq m \leq n$ satisfying:
\begin{enumerate} %\label{eq:deltaconditions}
    \item $\delta_{i,m} \geq 0$ for all $m$ whenever $v_i$ is maximal in $P$,\label{eq:deltaconditionsa}
    \item $\delta_{j,m} \leq \varepsilon_{j, i, m} + \delta_{i,m}$  for all  $m$ whenever  $v_i \lessdot v_j$,  and \label{eq:deltaconditionsb} 
    \item $\sum_{m=1}^n \delta_{j,m} = N_j $ \label{eq:deltaconditionsc}
  \end{enumerate}

Then we set
    \[
    \widetilde{\phi}_n\left(\prod_{m=1}^n \prod_{i=0}^d x_{i,m}^{\a_{i,m}/q}\right) = \prod_{m=1}^n \prod_{i=0}^d x_{i,m}^{\delta_{i,m}}
    \]
for $i=0, \ldots, d$ and $m = 1, \ldots, n$, noting that \eqref{eq:deltaconditionsa} and \eqref{eq:deltaconditionsb} ensure our image is in $S^{\otimes n}$.

This completes the construction of $\widetilde{\phi}_n$, but before we proceed with the proof let us make explicit the image of an arbitary monomial of $(S^{\otimes n})^{1/q}$.  Consider $\prod_{m=1}^n \prod_{i=0}^d x_{i,m}^{a_{i,m}/q}$ with $a_{i,m} \in \mathbb{Z}_{\geq 0}$ for each $i = 0, \ldots, d$ and $m = 1, \ldots, n$. We denote by $\{ \mu \} = \mu - \lfloor \mu \rfloor $ the fractional part of a real number $\mu$.  Setting $\a_{i,m} = q \{ \frac{a_{i,m}}{q}\}$ for each $i,m$ expresses our arbitrary monomial
    \[
    \prod_{m=1}^n \prod_{i=0}^d x_{i,m}^{a_{i,m}/q} = \left( \prod_{m=1}^n \prod_{i=0}^d x_{i,m}^{\lfloor a_{i,m}/q \rfloor}\right) \cdot \left(\prod_{m=1}^n \prod_{i=0}^d x_{i,m}^{\a_{i,m}/q} \right) 
    \]
    %as an $S$-multiple of the monomial basis where the inductive construction above is valid.  Thus, we conclude that
    so we have
    \begin{equation} \label{eq:phitildeformula}
    \widetilde{\phi}\left(\prod_{m=1}^n \prod_{i=0}^d x_{i,m}^{a_{i,m}/q} \right) = \prod_{m=1}^n \prod_{i=0}^d x_{i,m}^{\lfloor a_{i,m}/q \rfloor + \delta_{i,m}},
    \end{equation}
    %where the $\delta_{i,m}$ satisfy the conditions in \eqref{eq:deltaconditionsa}, \eqref{eq:deltaconditionsb}, and \eqref{eq:deltaconditionsc}.
    % \[
    % \begin{array}{rcl}
    % \sum_{m=1}^n \delta_{i,m} &=& n \lfloor\frac{N_i}{n} \rfloor + (\#\{ m \mid 1\leq m \leq n,\frac{m}{n} \leq \{ \frac{N_i}{n} \} \, \}) \medskip  \\ & = &
    % n \lfloor\frac{N_i}{n} \rfloor + n \{ \frac{N_i}{n}\} \medskip \\
    % & = & N_i
    % \end{array}.
    % \]
    }
where the $\delta_{i,m}$ are determined by the $\alpha_{i,m}$ as above. 

\subsection*{Construction of $\phi_n$} Having defined $\widetilde{\phi}_n$, we next show that it restricts to a $p^{-e}$-linear map on $R^{\otimes n}$.  In other words, we need to check that $\widetilde{\phi}_n((R^{\otimes n})^{1/q}) \subseteq R^{\otimes n}$.  As before, we proceed to verify this containment monomial by monomial.  A monomial of $(R^{\otimes n})^{1/q}$ has the form $\prod_{m=1}^n \prod_{i=0}^d x_{i,m}^{a_{i,m}/q}$ with $a_{i,m} \in \mathbb{Z}_{\geq 0}$ for each $i = 0, \ldots, d$ and $m = 1, \ldots, n$, and such that
\begin{equation} \label{eq:inRn1overq}
a_{i,m} \geq a_{j,m} \qquad \mbox{for all} \qquad v_i \lessdot v_j
\end{equation}
with $v_i, v_j \in \bar{P}$ and  $m = 1, \ldots, n$.

Given \eqref{eq:phitildeformula}, in order to show $\widetilde{\phi}_n(\prod_{m=1}^n \prod_{i=0}^d x_{i,m}^{a_{i,m}/q}) \in R^{\otimes n}$, we need to verify that \eqref{eq:inRn1overq} implies
\begin{equation}
    \label{eq:inRn}
    \left\lfloor \frac{a_{i,m}}{q} \right\rfloor + \delta_{i,m} \geq \left\lfloor \frac{a_{j,m}}{q} \right\rfloor + \delta_{j,m} \qquad \mbox{for all} \qquad v_i \lessdot v_j
\end{equation}
with $v_i, v_j \in \bar{P}$ and  $m = 1, \ldots, n$. As $\alpha_{i, m} \in [0, q-1]$ for all $i,m$, we have 
\[
\left\lceil \frac{\alpha_{j,m}}{q} - \frac{\alpha_{i,m}}{q}  \right\rceil = \varepsilon_{j,i,m}.
\]
Thus, (b) gives us
\begin{equation}
 \left\lceil \frac{\alpha_{j,m}}{q} -  \frac{\alpha_{i,m}}{q} \right\rceil \geq \delta_{j,m} - \delta_{i,m}
    \label{eq:deltaconditions}
\end{equation}
for all $i,j,m$. Setting $\alpha_{i,m} = q \{ \frac{a_{i,m}}{q}\}$ for each $i,m$ yields
\begin{equation}
\label{eq:alphas}
\frac{a_{i,m}}{q} = \left\lfloor \frac{a_{i,m}}{q} \right\rfloor + \frac{\alpha_{i,m}}{q}.
\end{equation}
In particular, dividing \eqref{eq:inRn1overq} by $q$ and substituting we have
\[
\left\lfloor \frac{a_{i,m}}{q} \right\rfloor - \left\lfloor \frac{a_{j,m}}{q} \right\rfloor \geq   \frac{\alpha_{j,m}}{q} -  \frac{\alpha_{i,m}}{q} .
\]
As the left side of this inequality is an integer, we may round up the right and use \eqref{eq:deltaconditions} to get
\[
\left\lfloor \frac{a_{i,m}}{q} \right\rfloor - \left\lfloor \frac{a_{j,m}}{q} \right\rfloor \geq   \left\lceil \frac{\alpha_{j,m}}{q} -  \frac{\alpha_{i,m}}{q} \right\rceil \geq \delta_{j,m} - \delta_{i,m}
\]
which in turn yields \eqref{eq:inRn} as desired.  

%\begin{proof}[Conclusion of the proof of \Cref{Prop:combinatorialproperty}]

It remains to show that $\phi \in \sD_e^{(n)}$ for all $n \in \mathbb{N}$.  To that end, fixing $n \in \mathbb{N}$ we need to construct a lift $\phi_n$ of $\phi$ to $R^{\otimes n}$ compatible with the diagonal. In other words, we require that the following diagram 
\begin{equation} \label{eq:liftdiagram}
\xymatrix{
(R^{\otimes n})^{1/q} \ar[r]^-{\phi_n} \ar[d]_{\Delta_n^{1/q}} & R^{\otimes n} \ar[d]^{\Delta_n} \\
R^{1/q} \ar[r]_{\phi} & R
}
\end{equation}
is commutative.  We will check that  $\phi_n = \widetilde{\phi}_n|_{(R^{\otimes n})^{1/q}}$ is such a lift.

Consider once again an arbitrary monomial of $(R^{\otimes n})^{1/q}$, with the  form $\prod_{m=1}^n \prod_{i=0}^d x_{i,m}^{a_{i,m}/q}$ with $a_{i,m} \in \mathbb{Z}_{\geq 0}$ for each $i = 0, \ldots, d$ and $m = 1, \ldots, n$ satisfying \eqref{eq:inRn1overq}.  We have
\begin{eqnarray*}
(\phi \circ \Delta_n^{1/q})\left( \prod_{m=1}^n \prod_{i=0}^d x_{i,m}^{a_{i,m}/q} \right) & = & \phi \left( \prod_{i=0}^d x_{i}^{\sum _{m=1}^n a_{i,m}/q} \right)
\end{eqnarray*}
is zero if $\sum_{m=1}^n \a_{i,m} \not \equiv r_i \mod q$ for some $1 \leq i \leq n$.  But this is Case 1 from the construction of $\widetilde{\phi}_n$, so that such monomials are necessarily in the kernel of $\phi_n$ and thus in the kernel of $\Delta_n \circ \phi_n$ as well.  Hence, having checked the commutivity of \eqref{eq:liftdiagram} in Case 1, we may assume going forward that we are in Case 2 where $\sum_{m=1}^n \a_{i,m} \equiv r_i \mod q$ for all $1 \leq i \leq n$.  Here, we have
\begin{eqnarray*}
(\phi \circ \Delta_n^{1/q})\left( \prod_{m=1}^n \prod_{i=0}^d x_{i,m}^{a_{i,m}/q} \right) & = & \phi \left( \prod_{i=0}^d x_{i}^{\sum _{m=1}^n a_{i,m}/q} \right) \\ & = & \prod_{i=0}^d x_{i}^{\lfloor \sum _{m=1}^n a_{i,m}/q \rfloor}
\end{eqnarray*}
and also
\begin{eqnarray*}
(\Delta_n \circ \phi_n)\left( \prod_{m=1}^n \prod_{i=0}^d x_{i,m}^{a_{i,m}/q} \right) & = & \Delta_n \left( \prod_{m=1}^n \prod_{i=0}^d x_{i,m}^{\lfloor a_{i,m}/q \rfloor + \delta_{i,m}} \right) \\ & = & \prod_{i=0}^d x_{i}^{ \sum _{m=1}^n (\lfloor a_{i,m}/q \rfloor + \delta_{i,m}) }
\end{eqnarray*}
so we see that \eqref{eq:liftdiagram} commutes provided that
\[
  \sum _{m=1}^n \left(\left\lfloor \frac{a_{i,m}}{q} \right\rfloor + \delta_{i,m}\right) = \left\lfloor \sum _{m=1}^n \frac{a_{i,m}}{q} \right\rfloor 
\]
for all $i$.  But this is an immediate consequence of \eqref{eq:deltaconditionsc}; we have
\begin{eqnarray*}
\sum _{m=1}^n \left(\left\lfloor \frac{a_{i,m}}{q} \right\rfloor + \delta_{i,m}\right) & = & \left( \sum _{m=1}^n \left\lfloor \frac{a_{i,m}}{q} \right\rfloor \right) + \left( \sum _{m=1}^n \delta_{i,m} \right) \\ & = & \left( \sum _{m=1}^n \left\lfloor \frac{a_{i,m}}{q} \right\rfloor \right) + \left( \left\lfloor \sum _{m=1}^n  \frac{\alpha_{i,m}}{q} \right\rfloor \right) \\ & = & \left\lfloor \sum _{m=1}^n \frac{a_{i,m}}{q} \right\rfloor 
\end{eqnarray*}
as desired.

On the other hand, suppose we have a map $\phi \in \sD_e^{(n)}$ with $\phi(z^{1/q}) = 1$. By \autoref{thm:monomialSplittings}, we may assume that $\phi$ sends monomials to monomials. Then for every choice of $\alpha_{i,m}$, there must be a choice of $\delta_{i,m}$ satisfying \eqref{eq:deltaconditionsa} and \eqref{eq:deltaconditionsc}, simply by choosing $\delta_{i,m}$ to be the exponents in the image of $\xalphaqn$, namely $\phi(\xalphaqn) = \prod_{m=1}^n \prod_{i=0}^d x_{i,m}^{\delta_{i,m}}$.  We will show \eqref{eq:deltaconditionsb} using the following lemma.

\begin{lemma}\label{equalityassumption}
Fix $\prod_{i = 1}^d x_i^{a_i/q} \in R^{1/q}$, and let $\alpha_i = \{a_i/q\}q$, and fix some $v_i \lessdot v_j$.  Then there exists a choice of $a_l', 0 \leq l \leq q$ such that for all $v_k \lessdot v_l$, we have 
\begin{align*}
a_k' &\geq a_l'  \text{ for all } v_k \lessdot v_l\\
\left\{\frac{a_k'}{q}\right\} &= \frac{\alpha_k}{q}, \text{ and } \\
\left\lfloor \frac{a_i'}{q} \right\rfloor &= \left\lfloor \frac{a_j'}{q} \right\rfloor + \epsilon_{j,i}
\end{align*}
where $\epsilon_{j,i} = 1$ if $\alpha_j > \alpha_i$ and $\epsilon_{j,i} = 0$ otherwise. 
\end{lemma}
\begin{proof}
Since $\prod_{i = 1}^d x_i^{a_i/q} \in R^{1/q}$, we have that $a_k \geq a_l$ for all $v_k \lessdot v_l$.  If $\lfloor a_i/q \rfloor = \lfloor a_j/q \rfloor$, then we must have that $\epsilon_{j,i} = 0$ and we are done by letting $a_i' = a_i$.  Otherwise, assume $\lfloor a_i/q \rfloor > \lfloor a_j/q \rfloor$.  Again, if $\lfloor a_i/q \rfloor = \lfloor a_j/q \rfloor + \epsilon_{i,j}$ then we are done, so assume $\lfloor a_i/q \rfloor > \lfloor a_j/q \rfloor + \epsilon_{i,j}$.  Let $c = \lfloor a_i/q \rfloor - \lfloor a_j/q \rfloor - \epsilon_{i,j}$.  Define 
\begin{align*}
a_k' &:= a_k + cq \text{ for } v_k \leq v_j, v_k \neq v_i \text{ and } \\
a_k' &:= a_k  \text{ for all other } k    
\end{align*}
Clearly $\{a_k'/q\} = \{a_k/q\} = \alpha_k/q$, and by construction we have
$$
\left\lfloor \frac{a_i'}{q} \right\rfloor = \left\lfloor \frac{a_j'}{q} \right\rfloor + \epsilon_{j,i}
$$
Note that since $v_i \lessdot v_j$, there is no $v_k$ such that $v_i \leq v_k \leq v_j$.  Then for any $v_k$ such that $v_i \lessdot v_k$, we have $a_i' \geq a_k'$ (since this holds for $i,j$ by construction and for all other $k$ this is just $a_i \geq a_k$).  Further, for any $v_k \leq v_j$ and $v_k \lessdot v_l$, we have $a_k' = a_k + cq \geq a_l + cq \geq a_l'$.
In particular, $\prod_{i=0}^d x_{i,m}^{a_{i}'/q} \in R^{1/q}$ also.
\end{proof}

Now fix some $m$ and some $v_i \lessdot v_j$. By Lemma \ref{equalityassumption} there is some choice of monomial $\xaqn$ such that
\begin{equation} \label{eq:equalityassumption}
\left\lfloor \frac{a_{i,m}}{q} \right\rfloor = \left\lfloor \frac{a_{j,m}}{q} \right\rfloor + \epsilon_{j,i,m}
\end{equation}
Then in particular, since
\begin{equation*}
    \phi\left(\xaqn\right) = \prod_{m=1}^n \prod_{i=0}^d x_{i,m}^{\lfloor a_{i,m}/q\rfloor + \delta_{i,m}} \in R^{\otimes n}
\end{equation*}
we must have
\begin{align*}
    \left\lfloor \frac{a_{j,m}}{q} \right\rfloor + \delta_{j,m} &\leq \left\lfloor \frac{a_{i,m}}{q} \right\rfloor + \delta_{i,m}
    \\
    \implies \delta_{j,m} &\leq \delta_{i,m} + \varepsilon_{j,i,m}
\end{align*}
so that \eqref{eq:deltaconditionsb} also holds as desired.

\end{proof}

%\section{Diagonal $F$-regularity for Hibi rings (Examples?)}

\section{Main results}

We say that the poset $P$ has no top nodes if for all $v_j \in P$ there exists a unique $v_i \in \bar{P}$ with $v_i \lessdot v_j$.

\begin{theorem}\label{thm:notopnodes}
If a Hibi ring is diagonally $F$-regular, then so is the ring we get by attaching one more node that covers a single element of the old poset. In particular, posets with no top nodes yield diagonally $F$-regular rings.
\end{theorem}
%\Janet{Should this proof should go here instead of in the appendix?  I can clean it up}
%\Dan{Yeah, the thing that I want to put in the appendix is the proof of theorem 3.1. However, I think we should break up this proof into two parts. My thinking is we should have a proposition that says "R is DFR if and only if for choices of alpha we can find delta such that ...", and this would go in section 3. Then section 4 should have the above theorem, but the proof would be much simpler. The advantage of doing it this way would be that we could refer to the proposition in section 3 when describing our counter-example. }
\begin{proof}
Let $R= \kay[\mathcal{I}(P)]$, where $P = \{v_1, \dots, v_d\}$, and suppose $R$ is diagonally $F$-regular.  Let $P' = P \cup \{v_{d+1}\}$ where $v_{d+1}$ covers some unique element in $P$.  By possible relabeling, let this element be $v_d$ so that $v_d \lessdot v_{d+1}$.  Then let $R' = \kay[\mathcal{I}(P')]$.  As usual, let $v_0$ be the minimal element of $\bar{P}$ (and so also of $\bar{P'}$).
Let $z' = \prod_{i=0}^{d+1} x_i^{r_i}$, where $r_i$ is the maximal length of any (upwards) path between $v_i$ and any maximal element in $\bar{P'}$, and let $z = \prod_{i=0}^{d} x_i^{r_i}$.  We note that this choice of $z, z'$ still satisfies $r_i > r_j$ for all $v_i \lessdot v_j$ in $\bar{P}$ and $\bar{P'}$ respectively, so that we can use Proposition \ref{Prop:combinatorialproperty} with this choice of $z, z'$.

Note that by Proposition \ref{Prop:combinatorialproperty}, it suffices to show that for any choice of $\alpha_{i,m} \in [0,q-1]$ with $0 \leq i \leq d+1$ and $1 \leq m \leq n$ such that $\sum_{m=1}^n \alpha_{i,m} \equiv r_i \pmod{q}$ there exists a choice of $\delta_{i,m}$ satisfying
\begin{enumerate} %\label{eq:deltaconditions}
    \item $\delta_{i,m} \geq 0$ for all $m$ whenever $v_i$ is maximal in $P$,
    \item $\delta_{j,m} \leq \varepsilon_{j, i, m} + \delta_{i,m}$  for all  $m$ whenever  $v_i \lessdot v_j$,  and 
    \item $\sum_{m=1}^n \delta_{j,m} = N_j $
  \end{enumerate}
where as before, $N_j :=\lfloor \sum_{m=1}^n \frac{\a_{j,m}}{q} \rfloor$ and $\varepsilon_{j,i,m} = \left\lceil \frac{\alpha_{j,m}}{q} - \frac{\alpha_{i,m}}{q} \right\rceil$.

Fix some $\alpha_{i,m} \in [0, q-1]$ such that $\sum_{m=1}^n \alpha_{i,m} \equiv r_i \pmod{q}$, with $0 \leq i \leq d+1$ and $1 \leq m \leq n$.  Then clearly we have such a choice of $\alpha_{i,m}$ for $0 \leq i \leq d$ as well.  Since $R$ is diagonally $F$-regular, we already have such a choice of $\delta_{i,m}$ satisfying \eqref{eq:deltaconditionsa}-\eqref{eq:deltaconditionsc} for $0 \leq i \leq d$.
We will specify $\delta_{d+1,m}$ while making use of the assumption that the added element $v_{d+1}$ covers a unique element $v_d$.  

    \begin{Claim*}
    $\sum_{m=1}^n \left( \lceil \frac{\a_{d+1,m}}{q} - \frac{\a_{d,m}}{q}\rceil + \delta_{d,m}\right) \geq N_{d+1}$.
    \end{Claim*}
    
    \begin{proof}[Proof of Claim.]
    We  have $\lceil \frac{\a_{d+1,m}}{q} - \frac{\a_{d,m}}{q}\rceil \geq  \frac{\a_{d+1,m}}{q} - \frac{\a_{d,m}}{q}$ for each $m = 1, \ldots, n$, and summing up gives
    \[
    \left( \sum_{m=1}^n \left\lceil \frac{\a_{d+1,m}}{q} - \frac{\a_{d,m}}{q}\right\rceil \right) + \left( \sum_{m=1}^n \frac{\a_{d,m}}{q} \right) \geq \left(\sum_{m=1}^n \frac{\a_{d+1,m}}{q}\right).
    \]
    Inequalities are preserved after rounding down, whence 
    \[
    \left( \sum_{m=1}^n \left\lceil \frac{\a_{d+1,m}}{q} - \frac{\a_{d,m}}{q}\right\rceil \right) + N_d \geq N_{d+1}.
    \]
    The claim now follows using the fact that $\sum_{m=1}^n \delta_{d,m} = N_d$.
    \end{proof}

    Note also that $$ \frac{\a_{d+1,m}}{q} - \frac{\a_{d,m}}{q} \geq - \frac{\a_{d,m}}{q} > -1$$ as $0 \leq \a_{d,m} < q$, so we have $\lceil \frac{\a_{d+1,m}}{q} - \frac{\a_{d,m}}{q}\rceil \geq 0$.  In particular, as each $\delta_{d,m} \geq 0$, the terms of the sum in the claim above are non-negative integers.  Thus, there is a choice of integers $\delta_{d+1,m}$ for $m = 1, \ldots, n$ so that we have a choice of $\delta_{i,m}$ for $1 \leq i \leq d+1$ which satisfies \eqref{eq:deltaconditionsa}-\eqref{eq:deltaconditionsc} and $R'$ is also diagonally $F$-regular.
\end{proof}

We also know the following result: 
\begin{theorem}[{\cite[Proposition 5.5]{SmolkinCarvajal}}]
Let $R$ and $S$ be two diagonally $F$-regular $\kay$-algebras. Then $R\otimes_\kay S$ is diagonally $F$-regular.
\label{thm:tensorDFR}
\end{theorem}

%\Janet{working on this-- Sept 5}
Combining the above two theorems, we get:
\begin{theorem}\label{thm:comparabletopnodes}
If every top node in a poset $P$ is comparable to every other top node, then the associated Hibi ring $R = \kay[\mathcal I(P)]$ is diagonally $F$-regular. 
\end{theorem}

We will use the following result: 
\begin{proposition}[c.f. \cite{ARAMOVA2000431}]
 Given two finite posets $P$ and $P'$, the tensor product of their associated Hibi rings is:
 \begin{equation*}
     \kay[\mathcal{I}(P)] \otimes_\kay \kay [\mathcal{I}(P')] = \kay[\mathcal{I}(\bar{P} \oplus P')]
 \end{equation*}
 where $\bar{P}$ is formed from $P$ by adding a unique minimal element, and $A \oplus B$ is the poset which is has the elements and relations of $A \cup B$, with the additional relations $y \leq x$ if $x \in A, y \in B$. 
 \label{prop:posetOfTensorProd}
\end{proposition}

\begin{proof}[Proof of Theorem \ref{thm:comparabletopnodes}]
We show that any poset which has the property that every top node is comparable to every other top node can be constructed by adding elements covering unique elements or taking tensor products of smaller Hibi rings with this property. We proceed by induction on the number $n$ of top-nodes in $P$. If $n=0$, then $R$ is diagonally $F$-regular by \autoref{thm:notopnodes}. 

Now suppose $n>0$ and that $P$ has the property that all of its top nodes are comparable. Label the top nodes $v_1, \dots, v_n$, so that $v_1 \leq \dots \leq v_n$ and let $v_n$ be the largest top node.  Let $S_{> v_i}, S_{< v_i},$ and $S_{\not\sim v_i}$ be the subsets of $P$ containing the elements greater than $v_i$, less than $v_i$ and not comparable to $v_i$ respectively (and the relations in each of these subsets inherited from $P$).  In particular, we have that none of the elements of $S_{>v_n}$ cover any of the elements of $S_{<v_n}$ or $S_{\not\sim v_n}$, since $S_{> v_n}$ contains no top nodes.  By construction, $P$ is just $(\overline{S_{> v_n}} \oplus  S_{< v_n}) \cup S_{\not\sim v_n}$ with an additional relation for every element in  $S_{\not \sim v_n}$; namely each element in this subset covers a unique element in $S_{< v_n}\cup S_{\not \sim v_n}$ since $S_{\not \sim v_n}$ does not contain any top nodes, so that the elements of $S_{\not\sim v_n}$ can be added to $(\overline{S_{> v_n}} \oplus  S_{< v_n})$ to form $P$ by simply adding elements covering unique elements.  

As $S_{< v_n}$ has only $n-1$ top nodes, we know that $\kay[\mathcal I(S_{< v_n})]$ is diagonally $F$-regular by the inductive hypothesis. Then $\kay[\mathcal I(\overline{S_{> v_n}} \oplus  S_{< v_n})]$ is diagonally $F$-regular by \autoref{thm:tensorDFR} and Proposition \ref{prop:posetOfTensorProd}. It follows that $R$ is diagonally $F$-regular by \autoref{thm:notopnodes}.
\end{proof}

\begin{example}
Consider the poset seen in \autoref{fig:dfrExample}; call it $Q$. The two top-nodes of $Q$, $p$ and $u_1$, are comparable to each other, so \autoref{thm:comparabletopnodes} predicts that $\kay[\mathcal I(Q)]$ is diagonally $F$-regular. Indeed, we can write $Q$ as $Q = (\bar{P_1} \oplus P_2) \cup  \{q\}$, where $P_1$ is the subposet $P_1 = \{v_1, v_2, v_3\}$ with no relations, $P_2$ is the subposet $P_2 = \{u_1, u_2, u_3, u_4\}$  with the relations $u_1 \geq u_3$, $u_1 \geq u_4$, $u_2 \geq u_4$, and $q$ covers the unique element $u_4$.

\begin{figure}
    \centering
  \begin{tikzpicture}
\tikzset{vertex/.style = {shape=circle,draw,minimum size=1 em}}
\node[vertex] (v1) at  (1,8) {$v_1$};
\node[vertex] (v2) at  (3,8) {$v_2$};
\node[vertex] (v3) at  (5,8) {$v_3$};
\node[vertex] (p) at  (3,6) {$p$};
\node[vertex] (u1) at  (3,4) {$u_1$};
\node[vertex] (u2) at  (5,4) {$u_2$};
\node[vertex] (q) at (7,4) {$q$};
\node[vertex] (u3) at (1,2)  {$u_3$};
\node[vertex] (u4) at (5,2) {$u_4$};
\node[vertex] (min) at (3,0) {$-\infty$};

\draw (v1) -- (p);
\draw (v2) -- (p);
\draw (v3) -- (p) -- (u1);
\draw (p) -- (u2) -- (u4);
\draw (u1) -- (u4);
\draw (u1) -- (u3);
\draw (q) -- (u4);
\draw (u4) -- (min);
\draw (u3) -- (min);

\end{tikzpicture}
    \caption{A poset with all top nodes comparable to one another}
    \label{fig:dfrExample}
\end{figure}
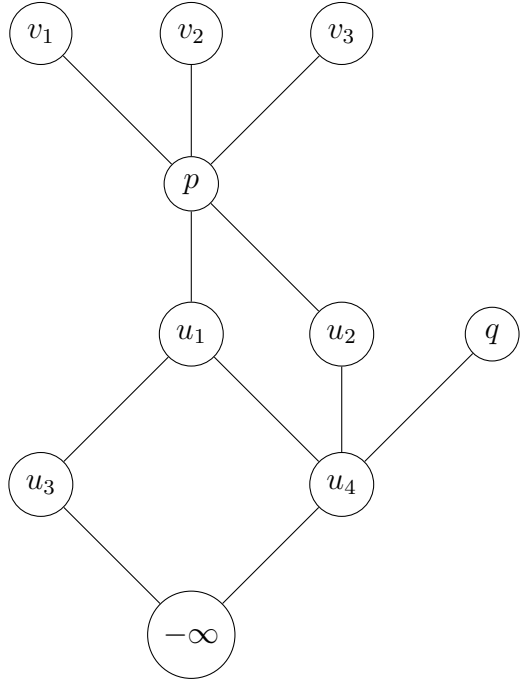

Note that $P_1 = (\bar{\varnothing} \oplus \{u_4,u_5\}) \cup \{u_2\}$ is constructed from the poset $\{u_4,u_5\}$ with no relations and the empty poset, and adding an element $u_2$ covering a unique element $u_4$. Thus we can see that $\kay[\mathcal I(Q)] $ is diagonally $F$-regular through repeated applications of \autoref{thm:notopnodes} and \autoref{thm:tensorDFR}.
\end{example}

%\Janet{I can add pictures, but what do you want here?  Examples of tensor products?}
%\Dan{Yeah, I was thinking we could have a picture of a poset satisfying the condition of the above theorem, and show how it breaks up into the tensor product of posets without top nodes}

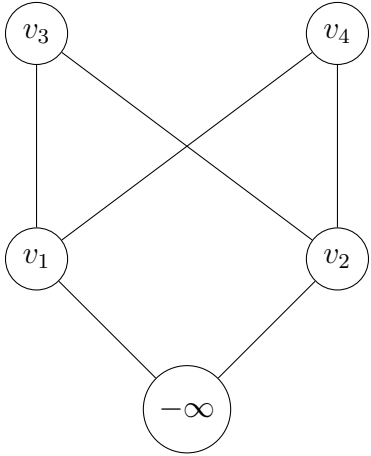
\begin{figure}[H]
    \centering
    \begin{tikzpicture}

\tikzset{vertex/.style = {shape=circle,draw,minimum size=1.5em}}
%\tikzset{edge/.style = {->,> = latex'}}
% vertices
\node[vertex] (a) at  (0,0) {$-\infty$};
\node[vertex] (b) at  (-2,2) {$v_1$};
\node[vertex] (c) at  (2,2) {$v_2$};
\node[vertex] (d) at  (-2,5) {$v_3$};
\node[vertex] (e) at (2,5) {$v_4$};
%edges
\draw (a) to (b);
\draw (a) to (c);
\draw (b) to (d);
\draw (b) to (e);
\draw (c) to (d);
\draw (c) to (e);
\end{tikzpicture}
    
    \caption{A poset for which the corresponding Hibi ring is not 3-diagonally $F$-regular in any characteristic.}
    \label{fig:4nodes}
\end{figure}
\begin{example}
\label{ex:not3diagFreg}
On the other hand, let $R$ be the Hibi ring over $\kay$ associated to the poset shown in  \autoref{fig:4nodes}. Then $R$ is the first Hibi ring that's not in the scope of \autoref{thm:comparabletopnodes}, and it is \emph{not} 3-diagonally $F$-regular in any characteristic. Indeed, using Proposition \ref{Prop:combinatorialproperty}, we can show there is no splitting of $t^2 x_1 x_2$ in $\sD^{(3)}_e(R)$ in any $e$. \autoref{fig:4nodes_eg} provides some choices of $\alpha_{i,m}$ for which the combinatorial problem of Proposition \ref{Prop:combinatorialproperty} is unsolvable, given $q > 2$. Note that, as a toric variety, the cone $\sigma$ corresponding to $R$ is given by the following extremal rays:
\[
\begin{pmatrix}
    1 \\
    -1 \\
    0 \\
    0 \\
    0
\end{pmatrix}, 
\begin{pmatrix}
    1 \\
    0 \\
    -1 \\
    0 \\
    0
\end{pmatrix}, 
\begin{pmatrix}
    0 \\
    1 \\
    0 \\
    -1 \\
    0
\end{pmatrix}, 
\begin{pmatrix}
    0 \\
    1 \\
    0 \\
    0 \\
    -1
\end{pmatrix}, 
\begin{pmatrix}
    0 \\
    0 \\
    1 \\
    -1 \\
    0
\end{pmatrix}, 
\begin{pmatrix}
    0 \\
    0 \\
    0 \\
    0 \\
    -1
\end{pmatrix}, 
    % \begin{array}{c}
    %      (1, -1, 0, 0, 0)  \\
    %      (1, 0, -1, 0, 0) \\
    %      (0, 1, 0, -1, 0) \\
    %      (0, 1, 0, 0, -1) \\
    %      (0, 0, 1, -1, 0) \\
    %      (0, 0, 1, 0, -1) \\
    %      (0, 0, 0, 1, 0)  \\
    %      (0, 0, 0, 0, 1) 
    % \end{array}
\]
whence one can see that $R$ has terminal singularities by \cite[Proposition 11.4.12]{CoxLittleSchenck}, if $\operatorname{char} \kay = 0$. We can also use \cite{WatanabeYoshidaMinimalRelHK} to see that the $F$-signature of $R$ is $11/30$, if $\operatorname{char} \kay > 0$.

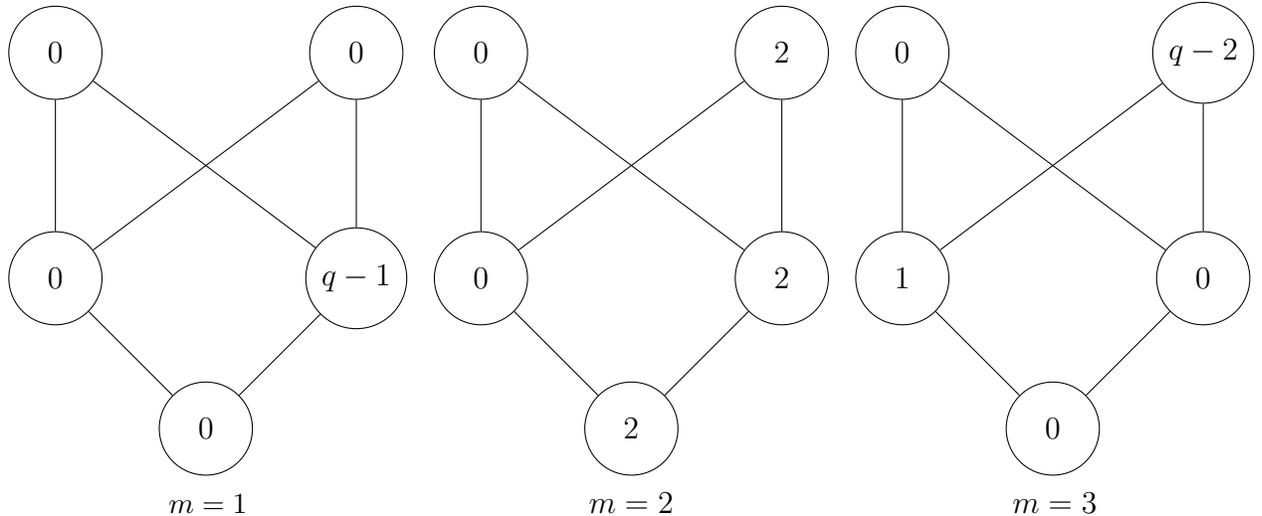
\begin{figure}[H]
    \centering
\begin{tabular}{ccc}
\begin{tikzpicture}
\tikzset{vertex/.style = {shape=circle,draw,minimum size=3em}}
% vertices
\node[vertex] (a) at  (0,0) {$ 0$};
\node[vertex] (b) at  (-2,2) {$ 0$};
\node[vertex] (c) at  (2,2) {$ q-1$};
\node[vertex] (d) at  (-2,5) {$ 0$};
\node[vertex] (e) at (2,5) {$ 0$};

\draw (a) -- (b);
\draw (a) -- (c); % node[midway, above]  {$ 1$};
\draw (b) -- (d);
\draw (b) -- (e);
\draw (c) -- (d);
\draw (c) -- (e);

\end{tikzpicture} &

\begin{tikzpicture}
\tikzset{vertex/.style = {shape=circle,draw,minimum size=3em}}
\node[vertex] (a2) at  (6,0) {$ 2$};
\node[vertex] (b2) at  (4,2) {$ 0$};
\node[vertex] (c2) at  (8,2) {$ 2$};
\node[vertex] (d2) at  (4,5) {$ 0$};
\node[vertex] (e2) at  (8,5) {$ 2$};

\draw (a2) -- (b2);
\draw (a2) -- (c2);
\draw (b2) -- (d2);
\draw (b2) -- (e2);% node[near end, above] {$ 1$};
\draw (c2) -- (d2);
\draw (c2) -- (e2);

\end{tikzpicture} &

\begin{tikzpicture}%[Here]
\tikzset{vertex/.style = {shape=circle,draw,minimum size=3em}}
\node[vertex] (a3) at  (12,0) {$ 0$};
\node[vertex] (b3) at  (10,2) {$ 1$};
\node[vertex] (c3) at  (14,2) {$ 0$};
\node[vertex] (d3) at  (10,5) {$ 0$};
\node[vertex] (e3) at  (14,5) {$ q-2$};

\draw (a3) -- (b3);% node[midway, above] {$ 1$};
\draw (a3) -- (c3);
\draw (b3) -- (d3);
\draw (b3) -- (e3);% node[near end, above] {$1$};
\draw (c3) -- (d3);
\draw (c3) -- (e3);% node[midway, right] {$ 1$};
\end{tikzpicture} \\
\small $m = 1$ & $m = 2$ & $m = 3$
\end{tabular}

    \caption{Choices of $\alpha_{i,m}$ for which the combinatorial condition described in  \autoref{Prop:combinatorialproperty} has no solution for $z = t^2 x_1 x_2$ and $q > 2$. %\Dan{Note that these numbers are a little different from what I showed you in Chicago---there was an issue before, where the alphas didn't add up the r's mod q. I'm pretty sure there are still no deltas that solve the combinatorial problem for this choice of alphas, but maybe you should double-check!}
    }
    \label{fig:4nodes_eg}
\end{figure}

% By this example, it follows the following rings in \autoref{fig:not_dfr} are also not 3-DFR in any characteristic. 

% \begin{figure}
%     \centering
% \begin{tabular}{cc}
% \begin{tikzpicture}
% \tikzset{vertex/.style = {shape=circle,draw,minimum size=.5em, fill}}
% %\tikzset{edge/.style = {->,> = latex'}}
% % vertices
% \node[vertex] (a) at  (0,0) {};
% \node[vertex] (b) at  (-2,1) {};
% \node[vertex] (c) at  (2,1) {};
% \node[vertex] (d) at  (-2,3) {};
% %\node[vertex] (e) at (2,5) {};
% \node[vertex] (f) at  (-2,5) {};
% \node[vertex] (g) at (2,5) {};
% %edges
% \draw (a) to (b);
% \draw (a) to (c);
% \draw (b) to (d);
% \draw (c) to (g);
% \draw (b) to (g);
% \draw (c) to (f);
% \draw (d) to (f);

% \end{tikzpicture} &

% \begin{tikzpicture}

% \tikzset{vertex/.style = {shape=circle,draw,minimum size=.5em, fill}}
% %\tikzset{edge/.style = {->,> = latex'}}
% % vertices
% \node[vertex] (a) at  (0,0) {};
% \node[vertex] (b) at  (-2,1) {};
% \node[vertex] (c) at  (2,1) {};
% \node[vertex] (d) at  (-2,3) {};
% \node[vertex] (e) at (2,3) {};
% \node[vertex] (f) at  (-2,5) {};
% \node[vertex] (g) at (2,5) {};
% %edges
% \draw (a) to (b);
% \draw (a) to (c);
% \draw (b) to (d);
% \draw (c) to (e);
% \draw (b) to (g);
% \draw (c) to (f);
% \draw (d) to (f);
% \draw (e) to (g);

% \end{tikzpicture}
% \end{tabular}

%     \caption{Posets corresponding to other hibi rings that are not 3-DFR}
%     \label{fig:not_dfr}
% \end{figure}
\end{example}

\appendix
\section{Proof of \autoref{thm:toricDn}}

This appendix will be spent proving \autoref{thm:toricDn}. The proof follows along the same lines as \cite[Theorem 6.4]{Smolkin}. We restate it here for convenience: 
\begin{theorem*}
  Let $R$ be a toric ring over $\kay$ and $X = \Spec R$. Set $n, e> 0$. Then $\sD^{(n)}_{e}(R)$ is generated as a $\kay$-vector space by the maps  $\pi_a \in \Hom_{R}(R^{1/q}, R)$ satisfying the following: for all sets of vectors $u_1, \ldots, u_n \in \frac{1}{q}\bZ^d$, there exist $v_i \in u_i + \bZ^d$ such that $v_i \in P_{-K_X}^\circ$ and $\sum_{i=1}^n v_i \in a - P_{-K_X}^\circ$. 
\end{theorem*}

We work in the following setting. 

\begin{setting}  We let $\Sigma$ be a rational cone in $\bR^d$ and $R = R(\Sigma)$ the associated toric ring over a perfect field $\kay$ of characteristic $p > 0$. Set $X = \Spec R$. For all rays $\rho \subseteq \Sigma$, we let  $v_\rho$ denote the primitive generator of $\rho$. That is, $v_{\rho}$ is the shortest nonzero vector in $\bZ^d \cap \rho$. Let $\Delta_n$ denote the diagonal embedding $X \to X^{\times n}$. Note that $\Delta_n$ induces the diagonal embedding $T \to T^{\times n}$, where $T\subseteq X$ is the dense torus contained in $X$. We write $\mathscr I(\Delta_n(X))$ to denote the sheaf of ideals in $\sO_{X^{\times n}}$ cutting out $\Delta_n(X)$ in $X^{\times n}$. Similarly, we write $\mathscr I(\Delta_n(T))$ to denote the sheaf of ideals in $\sO_{T^{\times n}}$ cutting out $\Delta_n(T)$. All tensor products are taken over $\kay$ and all fiber products are taken over $\Spec \kay$. For a sequence of vectors $a_1, \ldots, a_n \in \frac{1}{q}\bZ^d$, we set $c_{a_\bullet}\coloneqq c_{a_1, \ldots, a_n}$ and we set $\bm{\pi}_{a_\bullet} = \pi_{a_1 }\otimes \cdots \otimes \pi_{a_n}$.
\end{setting}

As $\kay$ is a field of characteristic $p > 0$, any $\kay$-scheme $Y$ comes with an \emph{absolute Frobenius morphism}, $F: Y\to Y$, defined as follows.  On the level of sets, $F$ is the identity. On the level of structure sheaves, we define:
\[ 
  F^\#: F_* \sO_Y \to \sO_Y, f\mapsto f^p.
\]
We denote the $e$-th iterate of $F$ by $F^e$. For any $\kay$-scheme $Y$, we write $\sC^Y_e \coloneqq \Hom_{\sO_Y} (F^e_* \sO_Y, \sO_Y)$.

\begin{definition}
Let $Y$ be a $\kay$-scheme, $\vp \in \sC^Y_e$, and $Z$ a closed subscheme of $Y$. We say $\vp$ is \emph{compatible} with $Z$ if $\vp(F^e_* \mathscr I(Z) ) \subseteq \mathscr I(Z)$, where $\mathscr I(Z)$ is the sheaf of ideals defining $Z$.
\end{definition}

  \begin{proposition}[{C.f.~\cite[Lemma 6.5]{Smolkin}}]
    Let $\displaystyle \varphi = \sum_{a_1, \ldots, a_n \in \frac{1}{q} \bZ^d} c_{a_\bullet} \bm{\pi}_{a_\bullet}$ be an element of $\sC^{T^{\times n}}_e$. Then $\varphi$ is compatible with $\Delta_n(T)$ if and only if, for all sequences of equivalence classes $[u_i], [v_i] \in \frac{1}{q} \bZ^d/\bZ^d$, $1 \leq i \leq n-1$, and all $D \in \frac{1}{q} \bZ^d$, we have
      \begin{equation*}
        \sum_{\substack{a_i \in [u_i]\\ a_n = D - \sum_{i=1}^{n-1} a_i}} c_{a_\bullet} = \sum_{\substack{b_i \in [v_i]\\ b_n = D - \sum_{i=1}^{n-1} b_i}} c_{b_\bullet}
      \end{equation*}
    \label{lemma:toricCompatWithDiag}
  \end{proposition}
  \begin{proof}
    As an $\sO_{T^{\times n}}$-module, $F^e_* \mathscr I(\Delta_n(T))$ is generated by the set
    \begin{equation*}
      \left\{ x^{v_1}\otimes \cdots \otimes x^{v_n} \left( x^{-u_1} \otimes x^{-u_2} \otimes \cdots \otimes x^{-u_{n-1}} \otimes x^{u_1 + \cdots + u_{n-1}} - 1 \right) \middle| v_i, u_i \in \frac{1}{q}\bZ^d \right\}. 
    \end{equation*}
    Thus, to say that $\varphi$ is compatible with $\Delta_n(T)$ is equivalent to saying that
    \begin{equation*}
      \varphi\left( x^{v_1}\otimes \cdots \otimes x^{v_n} \left( x^{-u_1} \otimes x^{-u_2} \otimes \cdots \otimes x^{-u_{n-1}} \otimes x^{\sum_{i=1}^{n-1} u_i} - 1 \right) \right) \equiv 0 \mod{\mathscr I (\Delta_n(T))}
    \end{equation*}
    for all $u_i, v_j \in \frac{1}{q} \bZ^d$, where $1 \leq i \leq n-1$ and $1 \leq j \leq n$.  In other words, 
    \begin{equation*}
      \varphi\left(  x^{v_1-u_1} \otimes x^{v_2-u_2} \otimes \cdots \otimes x^{v_{n-1}-u_{n-1}} \otimes x^{v_n + \sum_{i=1}^{n-1} u_i}\right) \equiv \varphi\left(  x^{v_1}\otimes \cdots \otimes x^{v_n} \right) \mod{\mathscr I(\Delta_n(T))}
    \end{equation*}
    for all $u_i, v_j \in \frac{1}{q} \bZ^d$. Using our description of $\vp$, this amounts to saying that
    \begin{align*}
      & \sum_{a_1, \ldots, a_n \in \frac{1}{q} \bZ^d} c_{a_\bullet} \bm{\pi}_{a_\bullet} \left(  x^{v_1-u_1} \otimes x^{v_2-u_2} \otimes \cdots \otimes x^{v_{n-1}-u_{n-1}} \otimes x^{v_n + \sum_{i=1}^{n-1} u_i}\right) \\
      \equiv & \sum_{b_1, \ldots, b_n \in \frac{1}{q} \bZ^d} c_{b_\bullet} \bm{\pi}_{b_\bullet} \left(  x^{v_1}\otimes \cdots \otimes x^{v_n} \right) \mod{\mathscr I(\Delta_n(T))}, \label{toricEquiv} \tag{$*$} 
    \end{align*}
    for all $u_i, v_j \in \frac{1}{q} \bZ^d$. By definition of $\bm{\pi}_{b_\bullet}$, we have 
    \begin{equation*}
      \sum_{b_1, \ldots, b_n \in \frac{1}{q} \bZ^d} c_{b_\bullet} \bm{\pi}_{b_\bullet} \left(  x^{v_1}\otimes \cdots \otimes x^{v_n} \right)  = \sum_{\substack{b_i \in -v_i +  \bZ^d,\\ 1 \leq i \leq n}} c_{b_\bullet} x^{v_1+ b_1} \otimes \cdots \otimes x^{v_n + b_n},
    \end{equation*}
    and similarly,
    \begin{align*}
      & \sum_{a_1, \ldots, a_n \in \frac{1}{q} \bZ^d} c_{a_\bullet} \bm{\pi}_{a_\bullet} \left(  x^{v_1-u_1} \otimes x^{v_2-u_2} \otimes \cdots \otimes x^{v_{n-1}-u_{n-1}} \otimes x^{v_n + \sum_{i=1}^{n-1} u_i}\right)\\
      =& \sum_{\substack{a_i \in u_i - v_i  + \bZ^d , \\ 1 \leq i \leq n-1 \\ a_n \in -v_n - \sum_{i=1}^{n-1} u_i + \bZ^d}} c_{a_\bullet} x^{a_1 + v_1 - u_1} \otimes \cdots \otimes x^{a_{n-1} + v_{n-1} - u_{n-1}} \otimes x^{a_n + v_n + \sum_{i=1}^{n-1} u_i}.
    \end{align*}
    It follows that equation \eqref{toricEquiv} is equivalent to saying, 
    \begin{equation}
      \sum_{\substack{a_i \in u_i - v_i  + \bZ^d , \\ 1 \leq i \leq n-1 \\ a_n \in -v_n - \sum_{i=1}^{n-1} u_i + \bZ^d}} c_{a_\bullet} x^{\sum_{j=1}^n v_j + a_j} = \sum_{\substack{b_i \in -v_j +  \bZ^d,\\ 1 \leq j \leq n}} c_{b_\bullet} x^{\sum_{j=1}^n v_j + b_j}, \label{toricSumDownstairs} \tag{$**$}
    \end{equation}
    for all $u_i, v_j \in \frac{1}{q} \bZ^d$. Note that, in the sum on the left, the condition that 
    \begin{equation*}
      a_n \in -v_n - \sum_{i=1}^{n-1} u_i + \bZ^d
    \end{equation*}
    is equivalent to saying that $\sum_{j=1}^n v_j + a_j \in \bZ^d$. 

    Equation \eqref{toricSumDownstairs} is an equality of Laurent  polynomials, so it can be checked degree-wise. Thus we have that $\vp$ is compatible with $\ker \Delta_n$ if and only if
    \begin{equation*}
      \sum_{\substack{a_i \in u_i - v_i  + \bZ^d , \\ 1 \leq i \leq n-1 \\ a_n = D - \sum_{j=1}^n v_j -  \sum_{i=1}^{n-1} a_i}} c_{a_\bullet}  = \sum_{\substack{b_i \in -v_i +  \bZ^d,\\ 1 \leq i \leq n-1\\ b_n = D - \sum_{j=1}^n v - \sum_{i=1}^{n-1} b_i}} c_{b_\bullet}
    \end{equation*}
    for all $u_i, v_j \in \frac{1}{q} \bZ^d,$ where $1 \leq  i \leq n-1$ and $1 \leq j \leq n$, and all $D \in \bZ^d$.  Making the substitutions $D' = D - \sum_{j=1}^n v_n$,  $u_i' = u_i - v_i$, and $v'_i = - v_i$, for $1 \leq i \leq n-1$, we notice that the preceding equation is equivalent to saying that
    \begin{equation*}
      \sum_{\substack{a_i \in u_i' + \bZ^d,\\ 1 \leq i \leq n-1\\ a_n = D' - \sum_{i=1}^{n-1} a_i}} c_{a_\bullet} = \sum_{\substack{b_i \in v'_i + \bZ^d,\\ 1 \leq i \leq n-1 \\ b_n = D' - \sum_{i=1}^{n-1} b_i}} c_{b_\bullet}
    \end{equation*}
    for all $u'_i, v'_i \in \frac{1}{q} \bZ^d$ and all $D' \in \frac{1}{q} \bZ^d$, where $1 \leq i \leq n-1$.

  \end{proof}

  \begin{lemma}
    Let $\displaystyle \varphi = \sum_{a_1, \ldots, a_n \in \frac{1}{q} \bZ^d} c_{a_\bullet} \bm{\pi}_{a_\bullet} \in \sC^{T^{\times n}}_e$ be a map compatible with $\Delta_n(T)$. Fix any set of equivalence classes $[u_1], \ldots, [u_{n-1}] \in \frac{1}{q}\bZ^d/\bZ^d$. Then the restriction of $\vp$ to $\Delta_n(T)$ is given by 
    \begin{equation*}
      \overline \vp = \sum_{D \in \frac{1}{q} \bZ^d}  \sum_{\substack{a_i \in [u_i]\\ a_n = D - \sum_{i=1}^{n-1} a_i }} c_{a_\bullet} \pi_D.
    \end{equation*}

    \label{lemma:toricRestrictionCalc}
  \end{lemma}
  \begin{proof}
    By \autoref{lemma:toricCompatWithDiag}, we know that the maps
    \begin{equation*}
      \sum_{\substack{a_i \in \frac{1}{q}\bZ^d \\\sum_{i=1}^n a_i = D}} c_{a_\bullet} \bm{\pi}_{a_\bullet}
    \end{equation*}
    are compatible with $\Delta_n(T)$ for all $D \in \frac{1}{q} \bZ^d$. Further, the restriction map is linear. Thus we may assume without loss of generality that
    \begin{equation*}
      \vp = \sum_{\substack{a_i \in \frac{1}{q}\bZ^d \\ \sum_{i=1}^n a_i = D}} c_{a_\bullet} \bm{\pi}_{a_\bullet}
    \end{equation*}
    for some $D\in \frac{1}{q} \bZ^d$. Let $v \in \frac{1}{q} \bZ^d$, so that $x^v \in \Gamma\left(T^{\times n}, \sO_{T^{\times n}}\right)$, and let $\vp|_{\Delta_n(T)}$ denote the restriction of $\vp$ to $\Delta_n(T)$. We compute: 
    \begin{align*}
      \vp|_{\Delta_n(T)}(x^v) &=  \left. \left( \sum_{\substack{a_i \in \frac{1}{q}\bZ^d \\ \sum_{i=1}^n a_i = D}} c_{a_\bullet} \pi_{a_\bullet}(x^v \otimes 1 \otimes \cdots \otimes 1) \right)\right|_{\Delta_n(T)}\\
      &=  \left. \left( \sum_{\substack{a_1 \in -v + \bZ^d\\ a_2, \ldots, a_{n-1} \in \bZ^d\\ D - \sum_{i=1}^{n-1} a_i \in \bZ^d}} c_{a_\bullet} x^{a_1 + v} \otimes x^{a_2} \otimes \cdots \otimes x^{a_{n-1}} \otimes x^{D - \sum_{i=1}^{n-1} a_i} \right)\right|_{\Delta_n(T)}\\
      &=  \begin{cases}
        \sum_{\substack{a_1 \in -v + \bZ^d\\ a_2, \ldots, a_{n-1} \in \bZ^d}} c_{a_1, \ldots, a_{n-1}, D - \sum_{i=1}^{n-1} a_i } x^{v+D}, & v + D \in \bZ^d\\
        0, & \textrm{otherwise}
      \end{cases} \\
      &=   \sum_{\substack{a_1 \in -v + \bZ^d\\ a_2, \ldots, a_{n-1} \in \bZ^d \\ a_n = D - \sum_{i=1}^{n-1} a_i}} c_{a_\bullet}  \pi_D(x^v)
    \end{align*}
  By \autoref{lemma:toricCompatWithDiag}, we know that 
  \begin{equation*}
     \sum_{\substack{a_1 \in -v + \bZ^d\\ a_2, \ldots, a_{n-1} \in \bZ^d \\ a_n = D - \sum_{i=1}^{n-1} a_i}} c_{a_\bullet} = \sum_{\substack{a_i \in [u_i]\\ a_n = D - \sum_{i=1}^{n-1} a_i }} c_{a_\bullet},
  \end{equation*}
  as desired. 
  \end{proof}

  \begin{lemma}
    Let $\vp \in \sC^{X^{\times n}}_e$. Then $\vp$ induces a map $\vp_T \in \sC^{T^{\times n}}_e$. Further, $\vp$ is compatible with $\Delta_n(X)$ if and only if $\vp_T$ is compatible with $\Delta_n(T)$. 
    \label{lemma:compatXiffcompatT}
  \end{lemma}
  \begin{proof}
    The first statement follows by the same argument as in \cite[Lemma 1.1.7(i)]{BrionKumarFrobeniusSplitting} (note that we don't assume $\vp$ is a splitting, but this is not necessary for that argument). If $\vp$ is compatible with $\Delta_n(X)$, then $\vp_T$ is clearly compatible with $\Delta_n(T) = \Delta_n(X) \cap T^{\times n}$. On the other hand, suppose that $\vp_T$ is compatible with $\Delta_n(T)$. We show $\vp$ is compatible with $\Delta_n(X)$ by an argument similar to \cite[Lemma 1.1.7(ii)]{BrionKumarFrobeniusSplitting}. In particular, $\Delta_n(X)$ is closed and $\Delta_n(T)$ is dense in $\Delta_n(X)$. It follows that the closure of $\Delta_n(X) \cap T$ in $X^{\times n}$ is $\Delta_n(X)$. Further, we know that $\vp\left(F^e_* \mathscr I\left( \Delta_n(X) \right)\right)$ is an ideal sheaf that vanishes along  $\Delta_n(T)$. It follows that $\vp\left(F^e_* \mathscr I\left( \Delta_n(X) \right)\right)$ vanishes along $\Delta_n(X)$, and so $\vp\left(F^e_* \mathscr I\left( \Delta_n(X) \right)\right)\subseteq \mathscr I\left( \Delta_n(X) \right)$, as desired. 
  \end{proof}

  \begin{lemma}
Let $\displaystyle \varphi = \sum_{a_1, \ldots, a_n \in \frac{1}{q} \bZ^d} c_{a_\bullet} \bm{\pi}_{a_\bullet}$ be an element of $\Hom_{\sO_{X^{\times n}}}(F^e_* \sO_{X^{\times n}}, \sO_{X^{\times n}})$. Then $\varphi$ is compatible with $\Delta_n(X)$ if and only if the maps
\begin{equation*}
  \sum_{\substack{a_1, \ldots, a_n \in \frac{1}{q} \bZ^d\\ a_1 + \cdots a_n = D}} c_{a_\bullet} \bm{\pi}_{a_\bullet}
\end{equation*}
are compatible with $\Delta_n(X)$ for all $D \in \frac{1}{q}\bZ^d$. It follows that $\sD^{(n)}(R)$ is generated as a $k$-vectorspace by the maps $\pi_a$ such that $\pi_a  \in \sD^{(n)}(R)$. 
    \label{cor:toric}
  \end{lemma}
  \begin{proof}
    If $X = T$, then the first assertion follows immediately from \autoref{lemma:toricCompatWithDiag}. The second assertion follows from the first, along with \autoref{lemma:toricRestrictionCalc}. The general case follows from the case where $X=T$ by \autoref{lemma:compatXiffcompatT}.

  \end{proof}

  \begin{proof}[Proof of \autoref{thm:toricDn}] 
    By \autoref{cor:toric}, we can reduce to checking that a map $\pi_b \in \Hom_{\sO_X}( F^e_* \sO_X, \sO_X)$ is an element of $\sD^{(n)}(X)$ if and only if, for all $u_1, \ldots, u_{n-1} \in \frac{1}{q} \bZ^d$, there exist $v_i \in u_i + \bZ^d$ with $v_i \in P_{_{-K_X}}^\circ$ for $i = 1, \ldots, n-1$ and with $a - \sum_{i=1}^{n-1} v_i \in P_{-K_X}^\circ$. 

    For the ``only if'' implication, let $\pi_b \in \sD^{(n)}_{e, X}(X)$. Then there exists a lifting, 
    \begin{equation*}
      \widehat \pi \coloneqq \sum_{a_1, \ldots, a_n \in \frac{1}{q} \bZ^d \cap P_{-K_X}} c_{a_\bullet} \bm{\pi}_{a_\bullet},
    \end{equation*}
    of $\pi_b$ to $\Hom_{\sO_X^{\times n}}( F^e_* \sO_{X^{\times n}}, \sO_{X^{\times n}})$. By \autoref{cor:toric}, we have that 
    \begin{equation*}
      \widehat \pi^\circ \coloneqq \sum_{\substack{a_1, \ldots, a_n \in \frac{1}{q} \bZ^d \cap P_{-K_X}\\ a_1 + \cdots + a_n = b}} c_{a_\bullet} \bm{\pi}_{a_\bullet},
    \end{equation*}
    is compatible with $\Delta_n(X)$, and its restriction to $T^{\times n}$ is compatible with $\Delta_n(T)$ by \autoref{lemma:compatXiffcompatT}. It follows from \autoref{lemma:toricRestrictionCalc} that the restriction of $\widehat \pi^\circ$ to $\Delta_n(X)$ is $\pi_b$, and also that
    \begin{equation*}
      \sum_{\substack{a_1, \ldots, a_n \in \frac{1}{q} \bZ^d \cap P_{-K_X}\\ a_i \in [u_i], 1 \leq i \leq n-1\\ a_1 + \cdots + a_n = b}} c_{a_\bullet} = 1
    \end{equation*}
    for all sequences of equivalence classes, $[u_1], \ldots, [u_{n-1}] \in \frac{1}{q} \bZ^d/\bZ^d$. But this is only possible if these sums are nonempty. 

    Conversely, let $U_1, \ldots, U_{(qd)^{n-1}}$ be all the sequences of $n-1$ elements in $\frac{1}{q} \bZ^d/\bZ^d$. Let $U_{i,j}$ denote the $j$th element of the sequence $U_i$. By assumption, there exist $a_{i,j} \in U_{i,j}$ for all $i,j$ with $1 \leq i \leq (qd)^{n-1}$ and $1 \leq j \leq n$, such that $a_{i,j} \in P_{-K_X}^\circ$ for all $i,j$, and $b - \sum_{j=1}^{n-1}a_{i,j}\in P_{-K_X}^\circ$ for all $i$. Then we set
    \begin{equation*}
      \vp = \sum_{i=1}^{(qd)^{n-1}} \pi_{a_{i,1}} \otimes \cdots \otimes \pi_{a_{i, n-1}} \otimes \pi_{b - \sum_{j=1}^{n-1} a_{i,j}}.
    \end{equation*}
    By \autoref{lemma:compatXiffcompatT} and \autoref{lemma:toricCompatWithDiag}, $\vp$ is compatible with $\Delta_n(X)$. By \autoref{lemma:toricRestrictionCalc}, $\pi_b$ is the restriction of $\vp$ to  $\Delta_n(X)$. In particular, $\pi_b \in \sD^{(n)}(X)$. 
  \end{proof}

\bibliographystyle{amsalpha}
\bibliography{References}

\end{document}